\documentclass[11, reqno]{amsart}

\usepackage{amsfonts}
\usepackage{amscd}
\usepackage{amsmath, mathrsfs, amssymb, mathtools}
\usepackage{amsthm}
\usepackage{setspace}
\usepackage{hyperref}
\usepackage{color}
\usepackage{epsfig}
\usepackage{here}
\usepackage{graphicx}
\usepackage[all]{xy}
\usepackage{psfrag}
\usepackage{graphicx,transparent}
\usepackage{enumerate}
\usepackage{caption}

\theoremstyle{plain}
\newtheorem{theorem}{Theorem}[section]
\newtheorem{lemma}[theorem]{Lemma}
\newtheorem{proposition}[theorem]{Proposition}
\newtheorem{corollary}[theorem]{Corollary}

\newtheorem{problem}[theorem]{Problem}
\theoremstyle{definition}

\newcommand{\modm}{\cal M}
\newcommand{\cal}{\mathcal}
\newcommand{\br}{\mathbb{R}}
\newcommand\cM{\mathcal{M}}
\newcommand\cD{\mathcal{D}}
\newcommand{\cd}{\mathcal{D}}
\newcommand{\wc}{\mathfrak{wc}}
\newcommand{\cb}{\mathcal{B}}
\newcommand{\cc}{{\mathcal C}}
\newcommand{\cp}{{\cal P}}

\newcommand{\Mbar}[2]{\overline{\mathcal{M}}_{{#1}, {#2}}}
\newcommand{\bq}{\mathbb{Q}}

\begin{document}

\title[A tale of two volumes]{A tale of two volumes of moduli spaces: Weil--Petersson and Masur--Veech}

\date{\today}


\author{Dawei Chen}
\address{Department of Mathematics, Boston College, Chestnut Hill, MA 02467, USA}
\email{dawei.chen@bc.edu}

\author{Scott Mullane}
\address{School of Mathematics and Statistics, University of Melbourne, Australia}
\email{mullanes@unimelb.edu.au}

\begin{abstract}
Weil--Petersson and Masur--Veech volumes measure the sizes of moduli spaces of Riemann surfaces equipped with hyperbolic and flat metrics, respectively. Over the past several decades, the computation of these volumes has inspired remarkable developments in combinatorial enumeration, intersection theory, and recursion relations. In this survey, we review key results, methods, open problems, as well as interesting parallels that emerge in the approaches to computing both types of volumes. 
\end{abstract}

\maketitle

\setcounter{tocdepth}{1}
\tableofcontents

\section{Introduction}
\label{sec:intro}

Moduli spaces arise throughout mathematics as parameter spaces for geometric objects sharing similar structures. When the objects in question carry a metric, the corresponding moduli spaces naturally admit measures that quantify their sizes by volumes. The study of these volumes brings together diverse themes---combinatorial enumeration, modular compactifications,  intersection theory, and recursion relations---and serves as a meeting point of geometry, dynamics,  combinatorics, and mathematical physics. 

This survey focuses on two fundamental notions of volume on moduli spaces of Riemann surfaces: the {\em Weil--Petersson} and {\em Masur--Veech volumes}. Weil--Petersson volumes are associated with hyperbolic metrics on Riemann surfaces, possibly with cusps, conical singularities, or geodesic boundaries. Masur--Veech volumes arise from flat metrics with conical singularities induced by holomorphic differentials on Riemann surfaces. We review the main ideas underlying their definitions and computations, trace key developments in their study, and compare the methods that have emerged in the two settings. Our aim is to highlight their analogies as well as their contrasts, and to encourage further interactions between these parallel theories. Since the topics covered in this survey lie at the intersection of several fields, we focus on conveying the geometric intuition behind the relevant concepts and ideas, while minimizing technical details without sacrificing precision, in order to maintain a natural and fluent exposition. We hope that readers from diverse backgrounds will find the article accessible and rewarding, even upon a quick reading.  

\subsection*{Acknowledgements} 

Part of this survey is based on a talk delivered by the first-named author at the Richmond Geometry Meeting 2021. The authors are grateful to Marco Aldi, Allison Moore, and Nicola Tarasca for organizing this wonderful conference series and for their kind invitation to contribute to the proceedings. The first-named author is support by the National Science Foundation under Grant DMS-2301030 and by Simons Travel Support for Mathematicians. The second-named author is  supported by DECRA Grant DE220100918 from the Australian Research Council.

\section{Weil--Petersson volumes}
\label{sec:WP}

In this section we provide an overview of the definition of, the methods to compute, relations between, and the high genus asymptotics of Weil--Petersson volumes of moduli spaces.

\subsection{Moduli of hyperbolic surfaces}

The moduli space of hyperbolic surfaces with labeled boundary components is defined as
\begin{align*}
\modm_{g,n}^{\mathrm{hyp}}(L_1&,...,L_n)\coloneqq  \Big\{(\Sigma,\beta_1,...,\beta_n)\mid \Sigma \text{ oriented hyperbolic surface},\\
&\text{ genus }\Sigma=g, \partial \Sigma=\sqcup\beta_j, 2\cosh(L_j/2)=-\mathrm{tr}(A_j)\Big\}/\sim
\end{align*} 
for $g,n$ satisfying $2g-2+n>0$ and $L_j\in\br_{\geq0}\cup i(0,2\pi)$, where $A_j\in {\rm PSL}(2,\br)$ represents the conjugacy class defined by the holonomy of the metric around the boundary component $\beta_j$.  The quotient is defined by isometries preserving each $\beta_j$ and its associated conjugacy class.  A boundary component is geometrically realized in the hyperbolic surface by a geodesic of length $L_j\in\br_{>0}$, a cusp when $L_j=0$, or a cone angle $\theta_j=|L_j|$ when $L_j=i\theta_j\in(0,2\pi)i$ is purely imaginary, giving an incomplete hyperbolic metric on the punctured surface.

Further, in the case that $\mathbf{L}\in i[0,2\pi)^n$, the Gauss--Bonnet theorem imposes the condition that
\begin{equation}\label{eqn:GB}
\sum_{j=1}^n\theta_j=\sum_{j=1}^n|L_j|<2\pi(2g-2+n)  
\end{equation}
for the moduli space to be non-empty. 

There is a natural homeomorphism between $\modm_{g,n}^{\mathrm{hyp}}(\mathbf{0})$ and $\modm_{g,n}$, the moduli space of genus $g$ curves with $n$ labeled points, that sends a hyperbolic surface to the corresponding conformal class with inverse map defined by uniformization. Mirzakhani showed that, via a grafting procedure, this generalizes to a homeomorphism between $\modm_{g,n}^{\text{hyp}}(\mathbf{L})$ and $\modm_{g,n}$ for $L_i\in \br_{\geq0}$. In the case that $\mathbf{L}\in i[0,2\pi)^n$, McOwen \cite{McOPoi} proved the existence of a unique conical metric $h(\mathbf{L})$ with prescribed cone angles $\theta_j=|L_j|$, $j=1,...,n$ in any fixed conformal class, providing the homeomorphism for these cases.

\begin{problem}
    \label{prob:mixed}
    In \S \ref{subsec:higherangles} the definition of $\modm_{g,n}^{\mathrm{hyp}}(\mathbf{L})$ is extended to include $L_j\in i\br_{\geq 0}$ which is geometrically realized as a conical singularity of angle $\theta_j=-iL_j$. While closely related recent work \cite{gonchar, chek, huangtel} has considered moduli of hyperbolic surfaces with different boundary behavior. An interesting question is to provide a geometrically meaningful definition of the moduli space for $L_j\in \mathbb{C}\setminus (\br_{\geq0}\cup i\br_{\geq 0})$ and extend the domain of definition of $\modm_{g,n}^{\mathrm{hyp}}(\mathbf{L})$. 
\end{problem}

\subsection{The Weil--Petersson metric and volume}\label{subsec:WPmetric}

The moduli space of hyperbolic surfaces  $\modm_{g,n}^{\text{hyp}}(\mathbf{L})$ comes equipped with a finite measure which can be constructed in a number of ways including:
\begin{itemize}
\item via Fenchel--Nielsen coordinates when they exist \cite{WolWei};
\item from the top power of a natural symplectic form $\omega^{WP}_\mathbf{L}$, defined via the Weil--Petersson construction of a Hermitian metric or via a natural symplectic form on the character variety \cite{GolInv}; or
\item the Reidemeister torsion of the natural complex defined by the associated flat $\text{PSL}(2,\br)$ connection \cite{WitQua}.
\end{itemize}  

We first consider the first two constructions.

\subsubsection{Fenchel--Nielsen coordinates} A \emph{pants decomposition} of a hyperbolic surface in $\modm_{g,n}^{\text{hyp}}(\mathbf{L})$ with $L_j\in\br_{\geq0}$ is a collection of geodesics on the surface that decompose the surface into genus zero surfaces with three geodesic boundary components. We denote the lengths of these geodesics by $\ell_1,\dots,\ell_{3g-3+n}\in\br_+$. A genus zero hyperbolic surface with three geodesic boundary components is uniquely determined by the lengths of the geodesics, though such a collection of genus zero surfaces is not sufficient to determine the surface in $\modm_{g,n}^{\text{hyp}}(\mathbf{L})$. The twist parameters $\tau_1,\dots,\tau_{3g-3+n}\in\br$ determine how each pair of geodesics are glued together and these two sets provide Fenchel--Nielsen coordinates for Teichm\"{u}ller space via a homeomorphism 
$$\mathcal{T}_{g,n}^{\text{hyp}}(L_1,...,L_n)\longrightarrow \br_+^{3g-3+n}\times \br^{3g-3+n}$$
under the natural topology on Teichm\"{u}ller space. The canonical symplectic form
$$\omega= \sum_{j=1}^{3g-3+n}d\ell_j\wedge d\tau_j  $$
then miraculously descends~\cite{WolWei} to a symplectic form on the moduli space $\modm_{g,n}^{\text{hyp}}(\mathbf{L})$ obtained as the quotient of Teichm\"{u}ller space under the action of the mapping class group. The resulting form is known as the Weil--Petersson symplectic form and denoted by $\omega_{\rm WP}({\bf{L}})$.

This can be generalized to define the Weil--Petersson symplectic form $\omega_{\bf{L}}$ for any choice of $\bf{L}$ such that the hyperbolic surfaces parameterised admit pants decompositions. By providing a McShane identity on hyperbolic surfaces with cone angles, it was shown in~\cite{TWZGen} that a pants decomposition exists when $L_j\in\br_{\geq0}\cup i[0,\pi)$ and hence this construction generalizes. However, in \cite{ANoVol} it was shown that when some $L_j\in i(\pi,2\pi)$, pants decompositions are no longer guaranteed in which case, a different path to defining the Weil--Petersson symplectic form is required.

\begin{problem}
\label{prob:coord}
Find a local coordinate system for $\modm_{g,n}^{\text{hyp}}(\mathbf{L})$ in the cases where pants decompositions and hence Fenchel--Nielsen coordinates do not exist.
\end{problem}

\subsubsection{The Weil--Petersson metric} 
The cotangent space at a point $(C,p_1,\dots,p_n)\in\modm_{g,n}$ is identified as
\[T^*_{[C,p_1,\dots,p_n]}\modm_{g,n}=H^1(C,T_C(-p_1-\cdots-p_n))^\vee\cong H^0(C,K^{\otimes 2}_C(p_1+\cdots +p_n)),\]
the vector space of meromorphic quadratic differentials with at worst simple poles on $\sum_{i=1}^np_i$. The Weil--Petersson metric $g_{\mathrm{WP}}$ and Weil--Petersson form $\omega_{\mathrm{WP}}$ are defined at such a point as the real and imaginary part of the Petersson pairing on the vector space of quadratic differentials:
\begin{equation}  \label{WPmetric}
\langle\eta,\xi\rangle\coloneqq  \int_C\frac{\overline{\eta}\xi}{h},\qquad \eta,\xi\in H^0 (C,K_C^{\otimes 2}(p_1+\cdots+p_n))
\end{equation}
where $h$ is the complete hyperbolic metric on $C\setminus\{p_1,\dots,p_n\}$.  The pairing \eqref{WPmetric} defines a Hermitian metric~\cite{WeiMod} on $\modm_{g,n}$ with imaginary part a K\"ahler form $\omega_{\rm WP}$.

This can be generalized~\cite{STrWei} by replacing the complete hyperbolic metric $h$ with a conical hyperbolic metric $h(\mathbf{L})$ with prescribed cone angles at points $p_1,\dots,p_n$. For  $\mathbf{L}\in i[0,2\pi)^n$ satisfying \eqref{eqn:GB} the K\"{a}hler metric $g_\mathbf{L}$ and associated K\"{a}hler form $\omega_{\rm WP}(\mathbf{L})$ are defined as the real and imaginary part of the Hermitian metric
\begin{equation}  \label{WPcmetric}
\langle\eta,\xi\rangle\coloneqq  \int_C\frac{\overline{\eta}\xi}{h(\mathbf{L})},\qquad \eta,\xi\in H^0(C,K_C^{\otimes 2}(p_1+\cdots +p_n)).
\end{equation}

The total measure is defined to be the volume of the moduli space:
\begin{equation}  \label{eq:vol}
\mathrm{Vol}\big(\cM^{\text{hyp}}_{g,n}(\mathbf{L})\big)=\int_{\modm_{g,n}^{\text{hyp}}(\mathbf{L})}\exp\omega_{\rm WP}(\mathbf{L}).
\end{equation}

\subsection{Wolpert: The volume for $\mathbf{L}=\mathbf{0}$ is finite }
The metric completion of $\modm_{g,n}$ with respect to $g_{\mathrm{WP}}$ defines $\overline{\modm}_{g,n}^{\mathrm{hyp}}(\mathbf{0})$. Points in the compactification correspond to nodal hyperbolic surfaces with cusps at nodes.   The role of complete nodal hyperbolic surfaces in the compactification of the moduli space was first studied by Bers \cite{BerSpa} and developed further in \cite{HarCha,HKoAna}.   

Wolpert~\cite{WolWei} showed the homeomorphism on the interior extends to a homeomorphism $\overline{\modm}_{g,n}^{\mathrm{hyp}}(\mathbf{0})\longrightarrow \overline{\modm}_{g,n}$ of the compactification by nodal hyperbolic surfaces with the Deligne--Mumford compactification of moduli space, and further, the symplectic form $\omega_{\rm WP}$ extends smoothly to a closed symplectic form on $\overline{\modm}_{g,n}^{\mathrm{hyp}}(\mathbf{0})$ via Fenchel--Nielsen coordinates and to a closed current via the complex-analytic structure on the Deligne--Mumford compactification $\overline{\modm}_{g,n}$. Both extensions were shown to represent the same cohomology class $\left[\omega_{\mathrm{WP}}\right]=2\pi\kappa_1\in H^2(\overline{\modm}_{g,n},\br)$. Hence
\begin{eqnarray*}
 \mathrm{Vol}\big(\cM^{\text{hyp}}_{g,n}(\mathbf{0})\big)&=&\int_{\modm_{g,n}^{\text{hyp}}(\mathbf{0})}\exp\omega_{\rm WP}\\
&=&\int_{\overline{\modm}_{g,n}^{\text{hyp}}(\mathbf{0})}\exp\omega_{\rm WP}\\
&=&\int_{\overline{\modm}_{g,n}}\exp\omega_{\rm WP}\\
 &=&\frac{(2\pi^2)^{3g-3+n}}{(3g-3+n)!} \int_{\overline{\modm}_{g,n}}\kappa_1^{3g-3+n} 
 \end{eqnarray*}
is finite.

\subsection{Mirzakhani: An innovative proof of Witten's conjecture} Mirzakhani 
 further showed that the volume of $\cM^{\text{hyp}}_{g,n}(\mathbf{L})$ for $\mathbf{L}\in \br_{\geq0}^n$ is finite, given by a symmetric polynomial $V^{\mathrm{Mirz}}_{g,n}(L_1,...,L_n)\in\br[L_1^2,...,L_n^2]$,  and provided two ways to compute the value:
\begin{enumerate}
\item \label{recpol} via a recursion \cite{MirSim} constructed by removing pairs of pants from hyperbolic surfaces with at least one boundary component that reduces the computation of the volumes to the initial calculations on one-dimensional moduli spaces $\mathrm{Vol}(\modm_{0,3}^{\mathrm{hyp}}(L_1,L_2,L_3))=1$ and $\mathrm{Vol}(\modm_{1,1}^{\mathrm{hyp}}(L))=\frac{1}{48}(L^2+4\pi^2)$ which were computed by N\"a\"at\"anen and Nakanish~\cite{NNAWei}; and 
\item \label{compol} generalizing Wolpert's work via a symplectic reduction argument \cite{MirWei}, Mirzakhani showed the homeomorphism on the interior extends to a homeomorphism 
$\overline{\modm}_{g,n}^{\mathrm{hyp}}(\mathbf{L})\longrightarrow \overline{\modm}_{g,n}$
of the compactification by nodal hyperbolic surfaces with the Deligne--Mumford compactification of moduli space, and showed that $\omega_{\mathrm{WP}}(\mathbf{L})$ extends to a closed current on $\overline{\modm}_{g,n}$, which represents the cohomology class 
\begin{align}  \label{WPdef}
\left[\omega_{\mathrm{WP}}(\mathbf{L})\right]& =\left[\omega_{\mathrm{WP}}\right]+\sum_{j=1}^n\tfrac12L_j^2\psi_j \\
& =2\pi^2\kappa_1 +\sum_{j=1}^n\tfrac12L_j^2\psi_j\in H^2(\overline{\modm}_{g,n},\br)
\end{align}
so that 
\begin{align}  \label{Poldef}
V^{\text{Mirz}}_{g,n}({\bf L})& \coloneqq  \int_{\modm_{g,n}^{\mathrm{hyp}}(\mathbf{L})}\exp\omega_{\mathrm{WP}}(\mathbf{L}) \\
& =\int_{\overline{\modm}_{g,n}}\exp\left(2\pi^2\kappa_1+\sum_{j=1}^n\tfrac12L_j^2\psi_j\right)\in\br[L_1,...,L_n].
\end{align}
\end{enumerate}
Taken together, these two results provided an innovative alternate proof of Witten's conjecture on the values of top $\psi$-intersections.  

Due to the connections and similarities to the story of Masur--Veech volumes of moduli space, the rest of this section will focus on work in the area that has followed in the following three directions: 
\begin{itemize}
\item
extending the volume function to cone points, of particular interest to physicists due to it's connection to Jackiw--Teitelboim (JT) gravity; 
\item
different recursions relating volumes of different moduli spaces and improving the efficiencies of computation of the volumes and intersection numbers; and
\item
the large genus asymptotics of the Mirzakhani volume polynomial, which have applications to random surfaces and intersection number asymptotics.
\end{itemize}

\subsection{Moduli of hyperbolic cone surfaces 
}

The study of hyperbolic surfaces with cone points was initiated by Schumacher and Trapani~\cite{STrVar} while Do and Norbury \cite{DNoWei} used the limiting volumes of these moduli spaces to construct new recursions in the Mirzakhani polynomials $V^{\text{Mirz}}_{g,n}({\bf L})$ (see \S \ref{S:recursion}).

The Mirzakhani polynomials $V^{\text{Mirz}}_{g,n}({\bf L})$ are polynomials in $L_j^2$ that Mirzakhani showed return the volume of the moduli space $\mathrm{Vol}\big(\cM^{\text{hyp}}_{g,n}(\mathbf{L})\big)$ when $\mathbf{L}\in \br_{\geq0}^n$.
When $L_j\in i[0,2\pi)$ the boundary component of the surface is realized as a cone point of angle $|L_j|$. The volume function on the parameter space $i[0,2\pi)^n$ is defined to return the volume of the corresponding moduli space 
\begin{eqnarray*}
\begin{array}{cccc}
\mathrm{Vol}\colon &i[0,2\pi)^n&\longrightarrow&\br_+\\
&\mathbf{L}&\mapsto& \mathrm{Vol}\big(\cM^{\text{hyp}}_{g,n}(\mathbf{L})\big).
\end{array}
\end{eqnarray*}
It is convenient to also use the parameters of the cone angles, that is, $\boldsymbol{\theta}=-i\mathbf{L}$ such that $\theta_j=-iL_j$. 

Tan, Wong, and Zhang~\cite{TWZGen} have provided a McShane identity on hyperbolic surfaces with cone angles which immediately generalizes Mirzakhani's recursion for $\mathbf{L}\in (\br_{\geq0}\cup(0,\pi)i)^n$. This shows that the polynomial $V^{\text{Mirz}}_{g,n}({\bf L})$ corresponds to the volume function $\mathrm{Vol}\big(\cM^{\text{hyp}}_{g,n}(\mathbf{L})\big)$ for $\mathbf{L}\in i[0,\pi)^n$. However, both the recursion and symplectic reduction arguments used by Mirzakhani depended vitally on the existence of pants decompositions of the hyperbolic surfaces, which does not exist for all $\mathbf{L}\in i[0,2\pi)^n$.
Observing the existence of values of $\mathbf{L}\in i[0,2\pi)^n$ that satisfy the Gauss--Bonnet theorem and hence correspond to non-empty moduli spaces, for which $V^{\text{Mirz}}_{g,n}({\bf L})<0$ showed the Mirzakhani polynomial does not extent to give the volume function $\mathrm{Vol}\big(\cM^{\text{hyp}}_{g,n}(\mathbf{L})\big)$ in all cases. Explicitly, it was observed~\cite{ANoVol},
$$ V^{\text{Mirz}}_{0,4}({0,0,i\theta,(2\pi-\varepsilon)i})<0  $$
for $0<4\pi\varepsilon<\theta^2<4\pi^2$.

By showing that the Weil--Petersson symplectic form extended smoothly to a closed current on the Deligne--Mumford compactification with the same cohomology class, Anagnostou and Norbury \cite{ANoVol} further showed that Mirzakhani's polynomial does provide the volume when $\theta_i+\theta_j=|L_i|+|L_j|<2\pi$ for all $i,j$, which includes cases where Fenchel--Nielsen coordinates do not exist, that is, when one angle $\theta_j\in(\pi,2\pi)$. 

Then Anagnostou, Norbury, and the second--named author \cite{AMN} showed the volume function is a piecewise polynomial in $L_j$ (or $\theta_j$) and explicitly described how the polynomial changes as one passes between the chambers on which the function is given by a fixed polynomial. These results are revisited in \S\ref{S:WC} after introducing the necessary machinery.

\subsubsection{Compactifications via metric completions: nodal hyperbolic surfaces}

For $\mathbf{L}\in i[0,2\pi)^n$ we define $\overline{\modm}_{g,n}^{\mathrm{hyp}}(\mathbf{L})$ the moduli space of nodal hyperbolic surfaces, as the metric completion of ${\modm}_{g,n}$ with respect to $g_{\mathbf{L}}$ defined in \S\ref{WPcmetric}. 

Setting $\mathbf{L}\geq \mathbf{L}'$, that is, $|L_j|\geq |L_j'|$ for all $j$, then for any pointed curve $[C,p_1,\dots,p_n]\in{\modm}_{g,n}$ the hyperbolic metric $h(\mathbf{L}')$ dominates $h(\mathbf{L})$ \cite{STrVar}, a consequence of which  \cite[Prop 2.2]{ANoVol} is a continuous map between the metric completions
$$ \overline{\modm}_{g,n}^{\mathrm{hyp}}(\mathbf{L}') \longrightarrow \overline{\modm}_{g,n}^{\mathrm{hyp}}(\mathbf{L}). $$
Setting $\mathbf{L}'=\mathbf{0}$ results in a continuous map
$$ \overline{\modm}_{g,n}\longrightarrow  \overline{\modm}_{g,n}^{\mathrm{hyp}}(\mathbf{L}) $$
showing the metric completions to be compact (shown previously in \cite{STrWei}) via a result of Masur~\cite{MasExt}, that the completion $ \overline{\modm}_{g,n}^{\mathrm{hyp}}(\mathbf{0})$ is homeomorphic to the Deligne--Mumford compactification $\overline{\modm}_{g,n}$. However, this map is not a homeomorphism in general. 
An immediate consequence of which, is that the strategy of Wolpert and Mirzakhani to compute the volume via intersection theory on an algebraic compactification that is homeomorphic to the metric completion will need a new candidate other than the Deligne--Mumford compactification  $\overline{\modm}_{g,n}$.

In \cite[\S 2.2]{ANoVol} the authors provided a study of the local behavior of the universal curves over the metric completion $\overline{\modm}_{g,n}^{\mathrm{hyp}}(\mathbf{L})$. Local hyperbolic geometry shows for any subset $S\subset \{1,\dots,n\}$ such that 
$$ \sum_{j\in S}\theta_j\geq 2\pi(|S|-1)  $$
the points of $S$ can come arbitrarily close in the universal curve over $\modm_{g,n}$, and hence collide in the limit. Any algebraic compactification of $\modm_{g,n}$ that is homeomorphic to $\overline{\modm}_{g,n}^{\mathrm{hyp}}(\mathbf{L})$ would necessarily also admit this behavior. As the Deligne--Mumford compactification $\overline{\modm}_{g,n}$ of ${\modm}_{g,n}$ does not admit this behavior, we are forced to look at alternate algebraic compactifications of ${\modm}_{g,n}$ for candidates that are homeomorphic to $\overline{\modm}_{g,n}^{\mathrm{hyp}}(\mathbf{L})$ to open the avenue of intersection theory on these algebraic compactifications to compute $\mathrm{Vol}({\modm}_{g,n}^{\mathrm{hyp}}(\mathbf{L}))$.

\subsubsection{Algebraic compactifications: Hassett stability}
Hassett \cite{HasMod} produced algebraic compactifications $\overline{\modm}_{g,\mathbf{a}}$ of $\cM_{g,n}$ using weighted stable curves depending on weights $\mathbf{a}=(a_1,...,a_n)\in(0,1]^n$ where the weights allow points to come together.

Hassett \cite{HasMod} defined the space of parameters 
\[ \cd_{g,n} \coloneqq  \{\mathbf{a}\in(0,1]^n\mid\sum a_i>2-2g\}\]
which define stability conditions governing a collection of permitted nodal curves used to produce alternative compactifications of $\cM_{g,n}$. For $\mathbf{a}\in\cd_{g,n}$, a pointed nodal curve $(C,p_1,\ldots,p_n)$  is $\mathbf{a}$-stable when the points $p_1,\ldots,p_n$ lie in the smooth locus of $C$, the twisted dualizing sheaf $\omega_C(a_1p_1+\cdots +a_np_n)$ is ample on every irreducible component of $C$, and $\{p_j=p\mid j\in J\}$ then $\displaystyle\sum_{j\in J}a_j\leq 1$. When $\mathbf{a}=(1^n)$ this is equivalent to the Deligne--Mumford stability condition. 

Hassett \cite{HasMod} showed $\overline{\modm}_{g,\mathbf{a}}$ of all  $\mathbf{a}$-stable curves defines a compactification  of $\cM_{g,n}$. Further, these compactifications are exactly what we require, as it is shown in \cite[Thm 1.3]{ANoVol}, that given $\mathbf{L}\in i[0,2\pi)^n$, defining
 \[ a(L_j)\coloneqq  1+\frac{iL_j}{2\pi}=1-\frac{\theta_j}{2\pi},
\]
and letting $\mathbf{a}=a(\mathbf{L})=(a(L_1),\dots,a(L_n))$ there exists a homeomorphism
\begin{equation}   \label{HassIso}
\overline{\modm}_{g,n}^{\mathrm{hyp}}(\mathbf{L})\cong\overline{\modm}_{g,\mathbf{a}}.
\end{equation}  
The possible homeomorphism types of these moduli spaces further decomposes according to the parameters. For any $J\subset\{1,...,n\}$ such that $|J|\geq 2$ define the \emph{wall}
\[W_J=\Big\{\mathbf{a}\in\cd_{g,n}\ \Big\vert\ \sum_{j\in J}a_j=1\Big\}.\]
The components of $\cd_{g,n}\setminus\cup_JW_J$ are {\em chambers} of $\cd_{g,n}$.  Given any $\mathbf{a}\in\cd_{g,n}\setminus\cup_JW_J$, 
for any $\mathbf{b}\in\cd_{g,n}$ that lies in the same chamber, a nodal curve $C$ is $\mathbf{a}$-stable if and only if it is $\mathbf{b}$-stable and hence $\overline{\modm}_{g,\mathbf{a}}\cong\overline{\modm}_{g,\mathbf{b}}$. 
We denote the unique Hassett compactification specified by a chamber $\cc$ by $\overline{\modm}_{g,\cc}$.

A chamber $\cc\subset\cd_{g,n}$ is uniquely determined by an order-preserving function on the power set $\cp(\mathbf{n})$,
\begin{equation}\label{Chamberdefn}
\cc\colon\cp(\mathbf{n})\to\{0,1\}\text{ defined on }J\subset\mathbf{n}\text{ by:}\quad \cc(J)=0\Leftrightarrow\forall\mathbf{a}\in\cc,\ \sum_{j\in J}a_j\leq1
\end{equation}
It is order-preserving with respect to the natural partial order on $\cp(\mathbf{n})$ and the total order $0<1$ on $\{0,1\}$. The sets $S\subset \{1,\dots,n\}$ such that $\cc(S)=0$ are the sets of points that can collide in the compactification.  
The \emph{main chamber} $\cc^\mathrm{M}\subset\cd_{g,n}$ contains the point $(1^n)$ and corresponds to the Deligne--Mumford compactification. 

A chamber $\cc\subset\cd_{g,n}$ incident to and above a wall $W_S$, so in particular $\cc(S)=1$, determines a unique chamber $\cc'$ that satisfies $\cc'(S)=0$ and $\cc'(J)=\cc(J)$ for all $J\neq S$.  Then $\cc'$ is incident to and below $W_S$, and $\cc$ is related to $\cc'$ by a simple wall-crossing.

For $\bold{a}=(a_1,\dots,a_n)\in \cd_{g,n}$, the classes  $\psi_j\in H^2(\Mbar{g}{\bold{a}},\bq)$ are defined  by 
$$\psi_j\coloneqq  c_1(s_j^*(\omega_\pi))$$ 
where $\omega_\pi$ is the relative dualizing sheaf of the universal family
$$\pi\colon \mathcal{C}_{g,\bold{a}}\longrightarrow \Mbar{g}{\bold{a}}$$
and $s_j$ is the section corresponding to the $j$th point.

We define $\kappa_1\coloneqq  \pi_*(c_1(\omega_\pi(\sum_{i=1}^n D_i))^2)$. Note that this definition follows that of~\cite{ACoCom} and differs from the Miller--Morita--Mumford class $\pi_*(c_1(\omega_\pi)^2)$ though both classes appear in the literature with the same notation.

In~\cite{AMN} the authors used the push-forward to define a class $\gamma(\bold{a})$:
$$\frac{1}{2}\pi_*(2\pi\omega_\pi(\mathbf{a}\cdot D))^2)=2\pi^2\gamma(\bold{a}),$$
which equivalently, defines the class $\gamma(\bold{a})\in H^2(\Mbar{g}{\cc},\bq)[\bold{a}]$ as
\begin{equation}  \label{gama}
\gamma(\bold{a})\coloneqq  \kappa_1-\sum (1-a_j)^2\psi_j+\sum_{\cc(i,j)=0}(2a_ia_j-2)D_{i,j}.
\end{equation}

The image of the formula \eqref{gama} under the pullback map into $H^*(\overline{\modm}_{g,n})$ produces the formula \cite[Equation (1.2)]{ETu2Dd} (up to extra $\psi$  classes representing geodesic boundary components that can also be included here).

Given a chamber $\cc\subset\cD_{g,n}$, define the volume polynomial $V_{g,\cc}$ by 
\begin{equation}  \label{volalg}
V_{g,\cc}(i\boldsymbol{\theta})\coloneqq  \int_{\Mbar{g}{\cc}}\exp(2\pi^2\gamma(\bold{a}))\in\br[\theta_1,...,\theta_n]
\end{equation}
where $\mathbf{a}_j=a(\theta_j)$. Further, the volume function can now be seen to be piecewise polynomial, where the function is a fixed polynomial over each Hassett chamber. That is,
\[V_{g,\cc}(i\boldsymbol{\theta})=\mathrm{Vol}\big(\cM^{\text{hyp}}_{g,n}(i\boldsymbol{\theta})\big),\quad \text{when }\mathbf{a}(\boldsymbol{\theta})\in\cc.\] 
The formula \eqref{volalg} agrees with Mirzakhani's formula \cite{MirWei} when $\cc=\cc^{\rm M}$ and generalizes it to all chambers.

\subsection{Wall-crossing volume polynomials}\label{S:WC}

Given a chamber $\cc\subset\cd_{g,n}$ incident to and above a wall $W_S$, hence uniquely determining $\cc'$ incident to and below $W_S$, we define the wall-crossing polynomial $\wc_{\cc,S}\in\br[\theta_1,...,\theta_n]$ by
\begin{equation}\label{WCis}
V_{g,\cc'}(i\boldsymbol{\theta}) = V_{g,\cc}(i\boldsymbol{\theta})+\wc_{\cc,S}(\boldsymbol{\theta}).
\end{equation}
 
For $S\subset\{1,\dots,n\}$, we let $\boldsymbol{\theta}_S\in[0,2\pi)^S$ denote the restriction of $\boldsymbol{\theta}\in[0,2\pi)^n$ to the entries indexed in $S$, and $S^c$ is the complement of $S\subset\{1,\dots,n\}$.   

Let $\cc\subset\cd_{g,n}$ be incident to and above a wall $W_S$, via intersection theoretic arguments the authors in \cite{AMN} proved 
\begin{equation} \label{WCF}
\wc_{\cc,S}(\boldsymbol{\theta})=\int_{0}^{\phi_S} V_{g,\cc/S}(i\boldsymbol{\theta}_{S^c},i\theta)\cdot V_{0,\cc_1}(i\theta,i\boldsymbol{\theta}_{S})  \cdot \theta \cdot d\theta
\end{equation}
where $\phi_S=\sum_{j\in S}\theta_j-2\pi(|S|-1)$ and the chambers $\cc/S$ and $\cc_1$ are identified as follows:
\begin{itemize}
\item
$\cc_1$ is the \emph{minimal chamber corresponding to the first entry} that can be identified as the unique chamber of of $\cD_{0,|S|+1}$ containing the element $(1,b^{|S|})$ for $\frac{1}{n-1}<b\leq\frac{1}{n-2}$.
\item 
$\cc/S$ is the \emph{quotient chamber} in $\cd_{g,n-|S|+1}$ defined by a natural quotient construction in \cite{AMN}, that can be identified as the unique chamber containing the element $(i\theta^{S^c},1)$.
\end{itemize}

\begin{problem}
\label{prob:wc}
Despite the seemingly geometrically suggestive form of the wall-crossing formula, including two volume polynomials of related moduli spaces, the proof is via intersection theory and hence explicitly algebro-geometric.

An open problem of interest is to give a geometric interpretation of the wall-crossing formula from the perspective of hyperbolic or symplectic geometry. 

Further, can the Weil--Petersson metric, volume function, wall-crossing phenomena and chamberwise polynomiality be extended to any expansion of the domain of $\mathbf{L}$ in the definition of the moduli space $\modm_{g,n}^{hyp}(\mathbf{L})$ (see Problem~\ref{prob:mixed}).   
\end{problem}

\subsection{Recursion and generalized string and dilaton equations}\label{S:recursion}

Mirzakhani's recursion sparked much interest in understanding relations between volumes of moduli spaces. There are three main ways to approach relations between volumes:
\begin{enumerate}
\item
Intersection theory: the cohomology classes of interest transform nicely under the morphism that forgets a marked point of $\overline{\modm}_{g,\mathbf{a}}$, which is further a flat morphism under certain conditions. This allows for the relation of the intersection calculations and hence volumes of different moduli spaces.
\item
Removing a pair of pants: the number of boundary components can be changed by removing a pair of pants (as in Mirzakhani's original work), which can be used to relate volumes of different moduli spaces.
\item
$2\pi$ limits: as the Weil--Petersson form degenerates, the volume function will approach zero as any cone point approaches $2\pi$.
\end{enumerate}

The polynomials $V^{\text{Mirz}}_{g,n}(\mathbf{L})$ were proven in \cite{DNoWei}, via relations between intersection numbers on $\overline{\modm}_{g,n}$, to satisfy the following relations:
\begin{subequations} \label{limitdil}
\begin{gather}
V^{\text{Mirz}}_{g,n+1}({\bf L},2\pi i)=\sum_{k=1}^n\int_0^{L_k}L_kV^{\text{Mirz}}_{g,n}({\bf L})dL_k, \label{limitdil1}\\
\frac{\partial V^{\text{Mirz}}_{g,n+1}}{\partial L_{n+1}}({\bf L},2\pi i)=2\pi i(2g-2+n)V^{\text{Mirz}}_{g,n}({\bf L}).\label{limitdil2}
\end{gather}
\end{subequations}
These relations generalize with via the wall--crossing perspective and give a geometric explanation of the suggestive $2\pi i$ limit. 
Let $\theta_j\to 2\pi$, equivalently $a_j\to0$, to define a path in $\cd_{g,n}$ along which walls are crossed only from above to below.  A chamber is:
\begin{itemize}
\item
\emph{flat in the $i$th coordinate} if there is a path $\theta_i\to 2\pi$ crossing only walls $W_S$ with $|S|=2$, and 
\item 
\emph{light in the $i$th coordinate} if there is a path $\theta_i\to 2\pi$ along which no walls are crossed.
\end{itemize}
If a chamber $\cc\subset\cd_{g,n+1}$ is light in the $(n+1)$th coordinate, then the chamber is incident to the $2\pi i$ boundary. As a consequence of a degeneration of the Weil--Petersson form in the $2\pi$ limit, the volume function and hence the polynomial corresponding to this chamber must vanish in this limit. That is, 
\begin{equation}\label{eqn:incidentzero}
V_{g,\cc}(i\theta_1,\dots,i\theta_{n},2\pi i)=0.
\end{equation}
This is also proven via an algebro-geometric argument via the transformation of the cohomology classes under  the forgetful morphism. 

Considering $\{\cb_j,S_j\}_{j=1}^k$ any series of simple wall-crossings in $\cd_{g,n+1}$ from $\cb=\cb_1$ to a chamber $\cb_{k+1}$ that is light in the $(n+1)$th coordinate, the vanishing of the volume function in the $2\pi$ limit gives, via the wall-crossing perspective, the following equality of polynomials:
\begin{equation}\label{eqn:incidentzerowc}
V_{g,\cc}(i\theta_1,\dots,i\theta_n,2\pi i)=-\sum_{j=1}^k\wc_{\cb_j,S_j}(\theta_1,\dots,\theta_n,2\pi )
\end{equation}
In particular, when applied to the main chamber $\cb=\cc^{\rm M}$ this produces the relation \eqref{limitdil1} via  \eqref{WCis}.

The equation \eqref{limitdil2} also generalizes to other chambers. A unique chamber $\cc|_{\mathbf{n}}\subset \cD_{g,n}$ is identified by a chamber $\cc\subset\cD_{g,n+1}$ that is flat in the $(n+1)$th coordinate, by forgetting the $(n+1)$th point. For any chamber $\cc\subset\cd_{g,n+1}$ flat in the $(n+1)$th coordinate
\begin{equation}\label{eqn:derivative}
\frac{\partial V_{g,\cc}}{\partial\theta_{n+1}}(i\theta_1,\dots,i\theta_n,2\pi i)=-2\pi \big(2g-2+|q(\cc)|+\sum_{j\notin q(\cc)} a(\theta_j)\big)V_{g,\cc|_{\mathbf{n}}}(i\theta_1,\dots,i\theta_n)
\end{equation}
where $q(\cc)=\{j\mid W_{\{j,n+1\}}\text{ is crossed as }\theta_{n+1}\to2\pi\}$.

For the main chamber $\cc^{\rm M}$ \eqref{eqn:derivative} reduces to \eqref{limitdil2}.  When $\cc$ is light in the $(n+1)$th coordinate, \eqref{eqn:derivative} gives the derivative of the piecewise polynomials $\mathrm{Vol}\big(\cM^{\text{hyp}}_{g,n}(i\boldsymbol{\theta})\big)$ proven previously in \cite[(1.3b)]{ETu2Dd}.

The result can be generalized via wall--crossing polynomial to chambers not flat in the $(n+1)$th coordinate \cite[Corollary 4.3]{AMN} and further, proofs of all results in \cite{AMN} allow the existence of boundary geodesics, provided the lengths of the added boundary components are sufficiently small (see \cite[Lemma 2.3]{ANoVol}).  

\subsection{Higher cone angles}\label{subsec:higherangles}
For general $\mathbf{L}\in i\br_{\geq0}^n$, allowing cone angles of $2\pi$ and greater, it is proven in \cite{McOPoi} that the existence of a hyperbolic metric with given cone angles in a given conformal class still holds and the space is endowed with a canonical symplectic form.  However, \eqref{WPcmetric} diverges whenever some $\theta_j=|L_j|\geq2\pi$ and hence it no longer defines a metric on $\modm_{g,n}$ and we lose the avenue of the metric completion to compactify ${\modm}_{g,n}^{\mathrm{hyp}}(\mathbf{L})$.

Nevertheless, Sauvaget~\cite{Sauv} has recently provided an innovative way forward, proposing a definition of the volume of these moduli spaces as a limit of the Masur--Veech volumes of moduli spaces of $k$-differentials and proved that the values correspond for $\mathbf{L}\in i[0,2\pi)^n$ in the case that $-iL_j=\theta_j<\pi$ for at least $n-1$ angles. We will provide more details of this in \S\ref{subsec:k-diff} and 
\S\ref{sec:common} as this direct connection between the two volumes of this tale most naturally sits in the discussion of the common themes of the two perspectives.

The work of Sauvaget and related developments provide motivation for a number of related questions of independent interest. Problem~\ref{prob:coord} to provide coordinates for ${\modm}_{g,n}^{\mathrm{hyp}}(\mathbf{L})$ in these cases is further motivated by these developments.
In addition, as discussed above, \eqref{WPcmetric} diverges whenever some $\theta_j=|L_j|\geq2\pi$ and it no longer defines a metric on $\modm_{g,n}$. Hence we cannot compactify ${\modm}_{g,n}^{\mathrm{hyp}}(\mathbf{L})$ via the metric completion.

\begin{problem}
    \label{prob:highangles}
    Progress the understanding of moduli of hyperbolic surfaces with higher cone angles. Open interesting problems include:
    \begin{enumerate}
    \item
    Verify that Sauvaget's definition of the volume coincides with the Weil--Petersson volume of these moduli spaces.
    \item
    Provide a suitable compactification of ${\modm}_{g,n}^{\mathrm{hyp}}(\mathbf{L})$ in the case that $\mathbf{L}\in i\br_{\geq0}^n$, allowing cone angles of $|L_j|>2\pi$. Further, provide a homeomorphic algebraic compactification of $\modm_{g,n}$ or prove that no such space can exist.
     \end{enumerate}
\end{problem}

\subsection{Large genus asymptotics}

\subsubsection{Early contributions} The contributions of various authors proved the existence of constants $0<c<C$ independent of $g$, such that for fixed $n$, asymptotically in $g$
$$c^g(2g)!< V^{\text{Mirz}}_{g,n}(\underline{0})<C^g(2g)!.$$
 Penner~\cite{PenDiff} gave an estimate of a lower asymptotic bound of $V^{\text{Mirz}}_{g,1}(0)$ via the introduction the decorated Teichm\"{u}ller space and a method of  integrating of top degree differential forms on $\modm_{g,n}$. Grushevsky~\cite{Grushevsky}  generalized this to obtain an upper asymptotic bound on $V^{\text{Mirz}}_{g,n}(\underline{0})$ for $n\geq1$. Schumacher and Trapani~\cite{STrInt} generalized Penner's lower bound via intersection theory to obtain a lower bound for $V^{\text{Mirz}}_{g,n}(\underline{0})$ for $n\geq0$.
 
 \subsubsection{The Mirzakhani--Zograf Expansion} Based on numerical data, Zograf~\cite{Zog} conjectured a precise asymptotic form for $V^{\text{Mirz}}_{g,n}$. 
 Following this, Mirzakhani~\cite{MirGrowth} proved 
 $$\frac{V^{\text{Mirz}}_{g,n}(\underline{0})}{V^{\text{Mirz}}_{g-1,n+2}(\underline{0})}\to 1$$ and $$\frac{V^{\text{Mirz}}_{g,n+1}(\underline{0})}{2gV^{\text{Mirz}}_{g,n}(\underline{0})}\to 4\pi^2$$ as $g\to \infty$. The two authors then combined ~\cite{MirZog}, and in a major contribution toward this conjecture showed that for fixed $n$, $V^{\text{Mirz}}_{g,n}(\underline{0})$ admits a full asymptotic expansion in inverse powers of $g$
 $$V^{\text{Mirz}}_{g,n}(\underline{0})=C_{\text{univ}}\frac{(2g-3+n)!(4\pi^2)^{2g-3+n}}{\sqrt{g}}\left(1+\frac{c_n^{(1)}}{g}+\cdots+\frac{c_n^{(k)}}{g^k}+O\left(\frac{1}{g^{k+1}}\right)\right)$$
 with $C_{\text{univ}}$ the universal constant conjectured to be $1/\sqrt{\pi}$ and $c_n^{(k)}$ degree $2k$ polynomials in $n$ with effectively computable coefficients. 
 
More generally, $V^{\text{Mirz}}_{g,n}(L_1,\dots,L_n)$ is polynomial in $L_i^2$, and understanding the asymptotics of these polynomials is critical for studying random surfaces.

\subsubsection{First-order approximation} Mirzakhani and Petri~\cite{MirPetri} proved 
$$\lim_{g\to\infty}\frac{V^{\text{Mirz}}_{g,n}(L_1,\dots,L_n)}{V^{\text{Mirz}}_{g,n}(\underline{0})}=\prod_{j=1}^n\frac{2\sinh(L_j/2) }{L_j}   $$
  and used this estimate to prove that the counting function of lengths of closed geodesics converges to a Poisson process. Recently, Huang~\cite{Huang} proved a conjecture of Kimura~\cite{Kimura} that if the lengths $L_i$ are allowed to grow with $g$, provided $L_i=o(\sqrt{g})$, the same formula holds up to a constant.
  
\subsubsection{Higher-order expansions} More recently, Anantharaman and Monk~\cite{AnMonk} extended the work of Mirzakhani and Petri to prove the existence of a full asymptotic expansion of ${V^{\text{Mirz}}_{g,n}(\bf{L})}/{V^{\text{Mirz}}_{g,n}(\underline{0})}$ and provided an explicit second-order term.
 
\subsubsection{Large $n$ asymptotics} While the focus has been on the asymptotics for large $g$, Manin and Zograf~\cite{ManinZograf} considered the case of $V^{\text{Mirz}}_{g,n}(\underline{0})$, while Hide and Thomas~\cite{HideThomas} extended this to $L_i\geq0$ and derived asymptotics for fixed $g$ in large $n$ and established formulas for intersections of $\psi$-classes in the large $n$ limit with behavior distinct from the large $g$ case. 

\subsubsection{The optimal spectral gap conjecture}
The spectrum of the Laplace--Beltrami operator consists of discrete eigenvalues 
$$0=\lambda_0<\lambda_1\leq \lambda_2\leq \cdots ,$$
where $\lambda_1$, the first non-zero eigenvalue controls important dynamical properties of the surface and is known as \emph{the spectral gap}.

The \emph{Optimal Spectral Gap Conjecture} states that for any $\epsilon>0$, the Weil--Petersson probability that $\lambda_1\geq 1/4-\epsilon$ tends to $1$ as $g\to \infty$. 

Roughly speaking, the Selberg trace formula relates the eigenvalues of the Laplace--Beltrami operator to the lengths of closed geodesics on the hyperbolic surface and the question of the asymtotics of the Weil--Petersson expectation of $\lambda_1$ is reduced to understanding the asymptotics of the volumes $V^{\text{Mirz}}_{g,n}(\bf{L})$.

Independently, Wu--Xue~\cite{WuXue316} and Lipnowski--Wright~\cite{LipWright} recently proved the bound of $3/16-\epsilon$ which was then improved to $2/9-\epsilon$ by Anantharaman and Monk~\cite{AnMonk}. The expansion of the volume polynomial allows for the computation of expectations of geometric functions which require high precision (of the order $1/g^N$ for large $N$) for spectral bounds.

\begin{problem}
    \label{prob:Spectral}
    Progress understanding the higher order asymptotics of the Mirzakhani polynomial to improve known bounds on the spectral gap. 
\end{problem}

\section{Masur--Veech volumes}
\label{sec:MV}

In this section, we provide a detailed overview of the definition of Masur--Veech volumes, the methods used to compute them, and their various generalizations. 

\subsection{Strata of holomorphic differentials}
\label{subsec:strata}

Masur--Veech volumes are defined on moduli spaces of holomorphic differentials on Riemann surfaces (also called Abelian differentials). Hence, we begin by reviewing their geometry.

Let $\mu = (m_1, \ldots, m_n)$ be an integral partition of $2g - 2$, where $m_i \ge 0$ for all $i$. Given a nontrivial holomorphic differential $\omega$ on a smooth Riemann surface $X$ of genus $g$, suppose that $\omega$ has $n$ distinct zeros at $z_1, \ldots, z_n$, where each $z_i$ is a zero of order $m_i$. Then we say that $\omega$ is of type $\mu$.

\subsubsection{Translation surfaces with cone points} 

There is a fascinating interplay between holomorphic differentials and translation structures with cone points. Away from the zeros of $\omega$, we can locally choose an analytic coordinate $u$ such that $\omega = du$. Viewing $u = x + {\rm i}y$ as the standard coordinate on the Euclidean plane, $\omega$ induces a flat metric on $X \setminus \{z_1, \ldots, z_n\}$. If $v$ is another local coordinate satisfying $\omega = dv$, then by the fundamental theorem of calculus, $u$ and $v$ differ by a translation, which preserves the induced flat metric. 

At a zero $z_i$ of order $m_i$, one can find a local coordinate $w$ such that $\omega = w^{m_i}dw$. Comparing this with the nearby expression $\omega = du$, we obtain $w = u^{m_i+1}$ up to a scalar factor. Consequently, as one travels once around $z_i$, the total angle under the flat metric induced by $du$ is $2\pi(m_i + 1)$. In this sense, $\omega$ induces a translation structure on $X$ with cone points at $z_1, \ldots, z_n$, where the cone angle at each $z_i$ equals $2\pi(m_i + 1)$.

\subsubsection{Period coordinates} 

Let $\mathcal H(\mu)$ be the moduli space parameterizing differentials $(X,\omega)$ of type $\mu$. Locally at $(X, \omega)$, the space $\mathcal H(\mu)$ admits a coordinate system given by the relative cohomology $H^1(X, \{z_1,\ldots, z_n\};\mathbb Z)$. More precisely, suppose 
$$\gamma_1,\ldots, \gamma_{2g}; \gamma_{2g+1},\ldots, \gamma_{2g+n-1}$$ 
form a basis for the relative homology $H_1(X, \{z_1,\ldots, z_n\};\mathbb Z)$, where the first $2g$ paths are closed and span the absolute homology $H_1(X; \mathbb Z)$, and each of the remaining paths $\gamma_{2g+i}$ joins $z_i$ to $z_n$ for $i = 1,\ldots, n-1$. Then 
$$\int_{\gamma_1}\omega,\ldots, \int_{\gamma_{2g}}\omega; \int_{\gamma_{2g+1}}\omega,\ldots, \int_{\gamma_{2g+n-1}}\omega$$ 
provide local coordinates for $\mathcal H(\mu)$, called {\em period coordinates}. From the viewpoint of the translation surface structure induced by $\omega$, deforming these period coordinates amounts to varying the parallel edges that bound the translation surface while preserving the number and cone angles of the singularities, thus determining a neighborhood of $(X,\omega)$ in $\mathcal H(\mu)$. As a consequence, $\mathcal H(\mu)$ is a complex orbifold of dimension 
$$\dim_{\mathbb C}\mathcal H(m_1,\ldots, m_n) = 2g+n-1,$$ 
where the orbifold structure arises from special $(X,\omega)$ with nontrivial automorphisms. 

Note that for fixed $g$, the union of $\mathcal H(\mu)$ over all partitions $\mu$ of $2g-2$ is the Hodge bundle $\mathcal H_g$ (with the zero section removed) of holomorphic differentials over the moduli space $\mathcal M_g$ of genus-$g$ Riemann surfaces. Shrinking a short period joining two zeros of orders $m_i$ and $m_j$ causes them to merge, which induces a natural stratification of $\mathcal H_g$ by the strata $\mathcal H(\mu)$, where $\mathcal H(\mu)$ contains $\mathcal H(\mu')$ in its boundary if and only if $\mu$ can specialize to $\mu'$ by merging its parts. In this sense, $\mathcal H(\mu)$ is called the {\em stratum} of holomorphic differentials of type $\mu$. 

\subsubsection{Connected components of strata of holomorphic differentials} 

It is worth noting that the strata of holomorphic differentials $\mathcal H(\mu)$ can be disconnected for certain $\mu$. In a seminal work, Kontsevich and Zorich classified their connected components completely. 

\begin{theorem}[A condensed version of {\cite[Theorems 1 and 2]{KZ03}}]
The strata $\mathcal H(\mu)$ of holomorphic differentials of type $\mu$ have up to three connected components, where additional components arise due to hyperelliptic and spin structures.
\end{theorem}

For example, if $\mu$ contains at least one odd entry and $\mu \neq (g-1, g-1)$, then $\mathcal H(\mu)$ is connected. On the other hand, the minimal stratum $\mathcal H(2g-2)$ for $g \geq 4$ consists of three connected components $\mathcal H(2g-2)^{\rm hyp}$, $\mathcal H(2g-2)^{\rm odd}$, and $\mathcal H(2g-2)^{\rm even}$. For $(X,\omega)\in \mathcal H(2g-2)^{\rm hyp}$, the curve $X$ is hyperelliptic, and the unique zero $z$ of $\omega$ is a Weierstrass point, i.e., a ramification point of the hyperelliptic double cover. For $(X,\omega)\in \mathcal H(2g-2)^{\rm odd}$, $X$ is non-hyperelliptic and $(g-1)z$ is an odd theta characteristic, i.e., $h^0(X, (g-1)z) \equiv 1 \pmod{2}$. For $(X,\omega)\in \mathcal H(2g-2)^{\rm even}$, $X$ is non-hyperelliptic and $(g-1)z$ is an even theta characteristic, i.e., $h^0(X, (g-1)z) \equiv 0 \pmod{2}$. From now on, when referring to a stratum of differentials, we often mean one of its connected components when the stratum is disconnected.

\subsection{The Masur--Veech volume of $\mathcal H(\mu)$} 
\label{subsec:MV-volume}

Recall the local period coordinates of $\mathcal H(\mu)$ introduced above. Let $Z_1, \ldots, Z_{2g+n-1}$ be a basis of
$H^1(X, \{z_1,\ldots, z_n\}; \mathbb C)$ induced by integrating $\omega$ against a basis of $H_1(X, \{z_1,\ldots, z_n\}; \mathbb Z)$, where $Z_j = X_j + {\rm i} Y_j$.
Consider the standard Lebesgue measure  
\begin{align}
\label{eq:MV-local}
 \nu = \prod_{j=1}^{2g+n-1} dX_j \wedge dY_j = \prod_{j=1}^{2g+n-1} \frac{{\rm i}}{2} dZ_j \wedge d\overline{Z}_j. 
 \end{align}
It is independent of the choice of a basis of $H_1(X, {z_1,\ldots, z_n}; \mathbb Z)$, since the change-of-basis matrix is integral with determinant $\pm 1$.

Since a differential parameterized in $\mathcal H(\mu)$ can be scaled arbitrarily, the above measure yields an infinite volume on the entire space $\mathcal H(\mu)$. To possibly obtain a finite volume, we consider the ``unit hyperboloid''
$\mathcal H(\mu)^{\leq 1}$ parameterizing translation surfaces $(X,\omega)$ with area $\leq 1$, where 
$${\rm Area} (X, \omega) = \frac{{\rm i}}{2}\int_{X} \omega \wedge \overline{\omega}. $$
We then define the {\em Masur--Veech volume} of $\mathcal H(\mu)$ to be  
$$ {\rm Vol}\,\mathcal H(\mu) \coloneqq \dim_{\mathbb R} \mathcal H(\mu)\cdot \int_{\mathcal H(\mu)^{\leq 1}} d\nu. $$

It is worth noting that even if we restrict to the locus $\mathcal H(\mu)^{\leq 1}$ of translation surfaces with bounded area, the space is not compact, and in principle there could be increasingly large volume contributions when approaching the boundary of the stratum. Hence, the finiteness of the volume on $\mathcal H(\mu)^{\leq 1}$ is a priori unclear. Nevertheless, Masur and Veech independently established the following finiteness result.  

\begin{theorem}[\cite{M82,V82}]
\label{thm:finite}
${\rm Vol}\,\mathcal H(\mu)$ is finite for all strata of holomorphic differentials. 
\end{theorem}

The remaining question is therefore to compute the Masur--Veech volume of $\mathcal H(\mu)$ and to study its numerical properties. 

\subsection{Computing Masur--Veech volumes via the enumeration of square-tiled surfaces}
\label{subsec:MV-ST}

Square-tiled surfaces correspond to integral points under the period coordinates in a stratum of holomorphic differentials. Therefore, their asymptotic growth rate can determine the volume of the stratum. 

\subsubsection{Square-tiled surfaces and branched covers of tori} 

Let $(E, q)$ be a flat torus with a marked point $q$, i.e., a complex elliptic curve, where the flat metric corresponds to a nowhere vanishing holomorphic differential $dz$. Given a partition $\mu = (m_1, \ldots, m_n)$ of $2g - 2$, consider connected degree-$d$ branched covers
$$\pi\colon X\to E$$ 
of genus $g$ such that the only branch point of $\pi$ is $q$, and the ramification type over $q$ is $\mu$, that is,
$$ \pi^{*}(q) = (m_1+1)z_1 + \cdots + (m_n+1)z_n + w_1 + \cdots + w_{d-(2g-2+n)}, $$
where the $z_i$ and $w_j$ are distinct. In other words, the ramification points of $\pi$ are $z_1, \ldots, z_n$, with each $z_i$ having ramification order $m_i$.

Note that this setup satisfies the Riemann--Hurwitz formula. In particular, the pullback differential
$\omega = \pi^{*}(dz)$ has zeros at $z_1, \ldots, z_n$, where each $z_i$ is a zero of order $m_i$. Therefore, $(X, \omega) \in \mathcal H(\mu)$.
From the viewpoint of translation surfaces, if $E$ is the unit-square flat torus with $q$ located at a vertex of the square, then the translation surface associated with $\omega$ tiles $X$ by $d$ unit squares, with the cone points $z_1, \ldots, z_n$ appearing at some of the vertices of these square tiles and the unramified preimage points $w_j$ of $q$ appearing at the remaining vertices. In this case, $(X, \omega)$ is called a {\em square-tiled surface}. The area of a square-tiled surface equals the number of unit squares that tile it, which coincides with the degree of the torus covering. 

The following result shows that square-tiled surfaces correspond to integral points in the strata of holomorphic differentials with respect to the period coordinates, and therefore play a fundamental role in the study of the moduli of differentials and Teichm\"uller dynamics.  

\begin{lemma}[{\cite[Lemma 3.1]{EO01}}]
Let $(X,\omega)$ be a holomorphic differential of type $\mu$,  parameterized by the stratum $\mathcal H(\mu)$. Then every period coordinate of $(X,\omega)$ lies in $\mathbb Z\oplus \mathbb Z{\rm i}$ if and only if $(X,\omega)$ is a square-tiled surface.  
\end{lemma}

\begin{proof}[Sketch of proof]
Let $\gamma$ be a path in $H_1(X, \{z_1,\ldots, z_n\}; \mathbb Z)$. If $(X,\omega)$ is a square-tiled surface, then  $$\int_{\gamma}\omega = \int_{\pi_{*}\gamma} dz. $$
Note that $\pi_{*}\gamma \in H_1(E, q; \mathbb Z)$, where $E$ is the unit-square torus. Therefore, $\int_{\gamma}\omega \in \mathbb Z \oplus \mathbb Z{\rm i}$.

Conversely, if $\int_{\gamma}\omega \in \mathbb Z \oplus \mathbb Z{\rm i}$ for all $\gamma \in H_1(X, \{z_1,\ldots, z_n\}; \mathbb Z)$, then 
$$\pi(x) = \int_{x_0}^x \omega$$ 
defines the desired covering map 
$\pi\colon X \to \mathbb C/ \mathbb Z\oplus \mathbb Z{\rm i} \cong E$, where $x_0\in X$ is a fixed base point. 
\end{proof}

\subsubsection{Large degree asymptotics of counting square-tiled surfaces} 

Let $C_d(\mu)$ denote the number of square-tiled surfaces of area $d$ in $\mathcal H(\mu)$, which coincides with the Hurwitz number of degree-$d$ connected torus coverings with a single branch point whose ramification type is $\mu$. Since $\mathcal H(\mu)$ is a cone, and a real scaling $t\omega$ multiplies the area of the flat surface by $t^2$, the Masur--Veech volume of $\mathcal H(\mu)$ is therefore determined by 
$$ \lim_{D\to \infty} D^{-\dim_{\mathbb C}\mathcal H(\mu)} 
\sum_{d=1}^D C_d(\mu)
$$
up to an explicit volume normalizing factor. 

As a starting point, Eskin and Okounkov showed that the generating function 
$$ C(\mu) = \sum_{d=0}^{\infty } q^d C_d(\mu) $$ 
is a quasimodular form in the variable $q$. They further developed a delicate algorithm to extract the leading coefficient of $C(\mu)$ using the $n$-point function and cumulants. In particular, they proved the following remarkable rationality result, originally conjectured by Kontsevich and Zorich.

\begin{corollary}[{\cite[Remark 6.8]{EO01}}]
\label{cor:EO-rational}
Let $\mathcal H(\mu)$ be the stratum of holomorphic differentials of genus $g$ and type $\mu$. Then 
$$ \frac{1}{\pi^{2g}} {\rm Vol}\, \mathcal H(\mu) \in \mathbb Q. $$
\end{corollary}

In Section~\ref{subsec:intersection}, we will see that the exponent $2g$ in the $\pi$-factor can be explained from the viewpoint of intersection theory. 

\subsection{Compactifications of the strata of differentials}
\label{subsec:compactification}

Before carrying out intersection calculations on any moduli space, a prerequisite is to compactify the moduli space, such that the compactified moduli space has a geometrically understandable boundary parameterizing degenerate elements with similar structures. 

There are several ways in which differentials $(X, \omega)$ parameterized by $\mathcal H(\mu)$ can degenerate. First, one can rescale $\omega$ to be arbitrarily large or small. This issue can be easily resolved by working with the projectivized stratum parameterizing canonical divisors of type $\mu$: 
$$\mathbb P\mathcal H(\mu) = \mathcal H(\mu)/\mathbb C^{*},$$ 
which is analogous to passing from a vector space to its projectivization. More seriously, if we deform the local period coordinates too much, the underlying translation surface may change its zero type---for instance, when two distinct zeros collide---or even fail to remain a smooth Riemann surface, as can happen when an absolute period shrinks to zero. To address these problems, we can work with Deligne--Mumford stable nodal curves whose marked points correspond to the zeros of the differentials. 

\subsubsection{Twisted differentials} 

Over the past decade, compactifying the strata of differentials has been a major breakthrough in the study of moduli of differentials and Teichm\"uller dynamics. A series of works, based on approaches from limit linear series, logarithmic geometry, and plumbing, all lead to the same type of degenerate objects known as twisted differentials. In brief, a {\em twisted differential} on a nodal curve $Y$ consists of $(Y_i, \eta_i)$, where each $Y_i$ is an irreducible component of $Y$ and $\eta_i$ is a possibly meromorphic differential on $Y_i$. At each node connecting two components $Y_i$ and $Y_j$, the differential $\eta_i$ has a zero on one side (say on $Y_i$) and $\eta_j$ has a pole on the other side, such that the zero and pole orders add up to $-2$. Here we assign a negative sign to the order of a pole to emphasize that it represents a pole rather than a zero. 

To understand how twisted differentials arise, consider a family of differentials $\omega_t$ degenerating from smooth curves $X_t$ to a nodal curve $Y$ as $t \to 0$. We assume that the zeros of $\omega_t$, viewed as marked points, converge to smooth points of $Y$ away from its nodes; in other words, we work with stable curves whose marked points correspond to the zeros of the differentials. However, the limit of $\omega_t$ may become identically zero on an irreducible component of $Y$, thus losing the information carried by that component. To address this issue, near each irreducible component $Y_i$ of $Y$, we can write 
$$\omega_t = t^{\ell_{Y_i}} \eta_t,$$ 
where $\ell_{Y_i} \in \mathbb N$ measures the vanishing rate of $\omega_t$ along $Y_i$. Equivalently, the twisted limit 
$$\eta_{Y_i} = \lim_{t\to 0} t^{-\ell_{Y_i}} \omega_t$$ 
is a non-identically-zero differential on $Y_i$. If $\ell_{Y_i} > \ell_{Y_j}$, this process produces poles for $\eta_i$ and zeros for $\eta_j$ at the nodes connecting $Y_i$ and $Y_j$. Their orders satisfy the relation coming from the local coordinate change $z_i z_j = t$, namely, $z_i^{m}dz_i \sim z_j^{-m-2}dz_j$ up to a scalar, so that $m + (-m - 2) = -2$. From the flat-geometric viewpoint, the subsurface around $Y_i$ shrinks faster than that around $Y_j$, making $Y_j$ appear larger—just as the flat neighborhood of a zero of order $m$ expands to a pole of order $-(m+2)$. Therefore, comparing these twisting exponents $-\ell_{Y_i}\in \mathbb Z^{\leq 0}$ induces a level structure on the dual graph of $Y$, in which the upper end of an edge corresponds to a zero of order $m$ of the twisted differential, and the lower end corresponds to a pole of order $-(m+2)$.  

\subsubsection{The global residue condition and multi-scale differentials} 

Conversely, given a twisted differential, one can ask whether it can be deformed into the interior of the stratum $\mathcal H(\mu)$. It turns out that there is one obstruction, called the {\em global residue condition}. In short, suppose $V = \{q_1,\ldots, q_k\}$ is a subset of vertical nodes of $Y$, whose highest polar lower ends lie on level $-\ell$, such that removing $V$ disconnects $Y$, i.e., the vanishing cycle $V_t$ near $V$ is a trivial homology class in $H_1(X_t; \mathbb Z)$. It follows that $\int_{V_t}\omega_t = 0$, and hence
\begin{align}
\label{eq:GRC}
\lim_{t\to 0}\int_{V_t} t^{-\ell} \omega_t = 0.
\end{align}
Since $\lim\limits_{t\to 0} (t^{-\ell} \omega_t)|_{Y_i} = \eta_{Y_i}$ for every component $Y_i$ on level $-\ell$, we obtain from \eqref{eq:GRC} that 
$$ \sum_{q_j^-\in Y_i} {\rm Res}_{q_j^-} \eta_{Y_i} = 0,$$
where $q_j^-$ denotes the polar lower end of $q_j$. 

Bainbridge, Gendron, Grushevsky, M\"oller, and the first-named author showed that a twisted differential of type $\mu$ can be smoothed into $\mathcal H(\mu)$ if and only if it satisfies the global residue condition; see \cite{BCGGM1}. Twisted differentials satisfying this condition are called {\em multi-scale differentials}, and their moduli space gives the most complete (and complicated) compactification of $\mathcal H(\mu)$ to date; see \cite{BCGGM2}. 

\subsubsection{A summary of strata compactifications} 

Depending on the motivation and application, several related compactifications have been studied, introducing additional boundary components and alternative interpretations of the global residue condition. To avoid obscuring our main goal, we simply refer to the following works: for compactifications without the global residue condition from the viewpoints of plumbing, limit linear series, and logarithmic geometry, see \cite{G18, C17, FP18, G16, CC19, CGHMS}; for interpretations of the global residue condition in tropical and Berkovich geometry as well as using Gorenstein singularities, see \cite{MMW21, TT22, BB24, CC24}; for a similar compactification of strata of $k$-differentials, see \cite{BCGGMk, CMZ24}; for a non-algebraic compactification from the flat-geometry viewpoint, see \cite{MW17, CW21}; for compactifications of linear subvarieties in the strata, see \cite{B23, BDG22}; and for an expository note comparing some of these compactifications, see \cite{D25}. 

From now on, we will simply denote by $\mathbb P \overline{\mathcal H}(\mu)$ the compactified stratum of multi-scale differentials of type $\mu$ constructed in \cite{BCGGM2}. The boundary of $\mathbb P \overline{\mathcal H}(\mu)$ consists of differentials of lower genera with more specialized singularity orders, providing a convenient setting for carrying out intersection calculations recursively. 

\subsection{Computing Masur--Veech volumes via intersection theory}
\label{subsec:intersection}

Just like the projectivization of a vector space, the projectivized stratum $\mathbb P\mathcal H(\mu)$ of holomorphic differentials of type $\mu$ carries a tautological holomorphic line bundle $\mathcal O(-1)$, whose fiber over a projective equivalence class $[(X, \omega)]$ is spanned by $\omega$. The dual bundle of $\mathcal O(-1)$, denoted by $\mathcal O(1)$, is called the universal line bundle. 

\subsubsection{The area hermitian metric} 

Let ${(Z_i, Z_{g+i})}_{i=1}^g$ be the basis of the space of absolute period coordinates obtained by integrating $\omega$ along a symplectic basis of $H_1(X; \mathbb Z)$. By the Riemann bilinear relations, the area of the translation surface $(X, \omega)$ can be expressed as follows: 
\begin{align*}
{\rm Area} (X, \omega) & =  \frac{\rm i}{2} \int_{X} \omega \wedge \overline{\omega} \\
& = \frac{{\rm i}}{2} \sum_{i=1}^{g} (Z_i \overline{Z}_{g+i} - \overline{Z}_i Z_{g+i}).  
\end{align*}

Since the area functional is positive, it induces a singular Hermitian metric $h$ on the line bundle $\mathcal O(-1)$. By the Chern–Weil theory, the curvature $(1,1)$-form 
$$ \varpi = \frac{-1}{2\pi {\rm i}}\partial \overline{\partial}\log h $$
represents the first Chern class of the universal line bundle, denoted by 
$$ \xi = c_1 (\mathcal O_{\mathbb P\mathcal H(\mu)}(1)).$$

Note that $\varpi^{2g-1}$ captures the component of absolute periods in the Masur--Veech volume form~\eqref{eq:MV-local}, where the exponent is taken to be $2g-1$ instead of $2g$ due to the projectivization setting. In particular, for the minimal strata $\mathbb P\mathcal H(2g-2)$, we can expect its Masur--Veech volume to be determined by the top self-intersection number $\xi^{2g-1}$, up to an explicit normalization factor in volume.

However, to make the above idea rigorous, one obstacle is to show that the Hermitian metric induced by the area form remains good when extended to the boundary of the compactified strata, so that the corresponding curvature form still represents $\xi$. This technical issue was later justified by Costantini, M\"oller, and Zachhuber through explicit coordinate computations using the multi-scale compactification, and was addressed more generally by Nguyen from the viewpoint of variations of Hodge structures; see \cite{CMZ24, N25}. 

\subsubsection{Masur--Veech volumes as intersection numbers} 

Modulo the issue of extending to the boundary, Sauvaget established the following intersection-theoretic formula for the Masur--Veech volume of the minimal strata. 

\begin{proposition}[{\cite[Proposition 1.3]{S18}}]
\label{prop:minimal}
For $g\geq 1$, we have 
$$ {\rm Vol}\,\mathcal H(2g-2) = \frac{2(2\pi {\rm i})^{2g}}{(2g-1)!} \int_{\mathbb P\overline{\mathcal H}(2g-2)} \xi^{2g-1}. $$
\end{proposition}

It is worth noting that the area form does not capture the variations of relative periods joining two distinct zeros of $\omega$. Therefore, to establish an intersection-theoretic formula for the Masur--Veech volume of a stratum of holomorphic differentials with more than one zero, besides $\xi^{2g-1}$, we must utilize additional divisor classes that record the relative positions of the zeros. As seen in Witten's conjecture, the $\psi$-classes $\psi_i$ associated with the zeros $z_i$ of $\omega$ as marked points are natural candidates for this purpose. Indeed, M\"oller, Sauvaget, Zagier, and the first-named author established the following general formula.

\begin{theorem}[{\cite[Theorem 1.1]{CMSZ20}}]
\label{thm:abelian-volume-int}
For any stratum of holomorphic differentials with signature $\mu = (m_1,\ldots, m_n)$, we have 
\begin{align}
\label{eq:MV-abelian}
{\rm Vol}\, \mathcal H(\mu) 
= \frac{2 (2\pi {\rm i})^{2g}}{(2g-3+n)!} \int_{\mathbb P\overline{\mathcal H}(\mu)} \frac{1}{m_i+1} \xi^{2g-1}\psi_1\cdots \widehat{\psi}_i \cdots \psi_n 
\end{align}
for all $1\leq i\leq n$. 
\end{theorem}

Before elaborating on the proof of \eqref{eq:MV-abelian}, we make several observations. First, since the divisor classes are $\mathbb Q$-Cartier, their intersection numbers are rational. Hence, the above formula provides another explanation of the rationality result 
$$\frac{1}{\pi^{2g}} {\rm Vol}\, \mathcal H(\mu)\in \mathbb Q,$$ originally proved by Eskin and Okounkov through the enumeration of square-tiled surfaces, as reviewed in Corollary~\ref{cor:EO-rational}. 

Next, for the $\psi$-class product in \eqref{eq:MV-abelian}, omitting $\psi_i$ amounts to dividing by $m_i+1$. These choices lead to the same value essentially because of the divisor class relations
$$(m_i+1)\psi_i\sim (m_j+1)\psi_j$$ 
modulo the boundary divisor classes of the compactified strata; see \cite[Proposition 2.1 (i)]{C19}. These relations are intuitively reasonable, since $\omega$ can be locally represented as $d(z_i^{m_i+1})$, and the $\psi_i$-bundle corresponds to the cotangent space generated by $dz_i$ at each marked zero $z_i$.  

\subsubsection{The quest for the flat geometric meaning of $\psi$} 

In order to prove \eqref{eq:MV-abelian}, a natural idea would be to find a Hermitian metric on each $\psi_i$-bundle and show that the product of their curvature forms captures the component of relative periods in the Masur--Veech volume form \eqref{eq:MV-local}. Unfortunately, this approach has not yet been carried out successfully. Instead, the proof in \cite{CMSZ20} proceeds in a rather indirect way: the authors showed  that the intersection numbers on the right-hand side of \eqref{eq:MV-abelian} satisfy certain recursive relations, which also hold for the Masur--Veech volumes of the strata obtained from large degree  asymptotics of enumerating square-tiled surfaces. Since both quantities coincide on the minimal strata $\mathcal H(2g-2)$, as seen in Proposition~\ref{prop:minimal}, the two systems of quantities---the Masur--Veech volumes and the intersection numbers in \eqref{eq:MV-abelian}---therefore coincide for all strata $\mathcal H(\mu)$. 
We summarize the open problems in this discussion as follows. 

\begin{problem}
\label{prob:psi}
Let $\psi_i$ denote the $\psi$-bundle associated with the $i$th zero of differentials in $\mathcal H(\mu)$. Construct Hermitian metrics on the bundles $\psi_i$ such that the product of their curvature forms corresponds to the component of relative periods in the Masur--Veech volume form~\eqref{eq:MV-local}. 

Furthermore, investigate the extension of these Hermitian metrics to the boundary of the multi-scale compactification. As a consequence, derive a direct proof of the intersection-theoretic formula~\eqref{eq:MV-abelian} for computing the Masur--Veech volume of~$\mathcal H(\mu)$. 
\end{problem}

\subsubsection{Large genus asymptotics of Masur--Veech volumes} 

Based on numerical experiments, Eskin and Zorich formulated a precise conjectural description of the large genus asymptotics of Masur--Veech volumes for the strata of holomorphic differentials. Using volume recursions expressed through intersection numbers, this conjecture was proved by M\"oller, Sauvaget, Zagier, and the first-named author, while an alternative proof was provided around the same time by Aggarwal via a delicate combinatorial analysis of the algorithm developed by Eskin and Okounkov. 

\begin{theorem}[A condensed version of {\cite[Main Conjecture 1]{EZ15}, \cite[Theorem 1.4]{A20}, \cite[Theorem 1.5]{CMSZ20}}]
\label{thm:MV-large}
The large genus asymptotics of Masur--Veech volumes of the strata of holomorphic differentials satisfy  
$$ \lim_{g\to \infty} (m_1+1)\cdots (m_n+1) {\rm Vol}\,\mathcal H(m_1,\ldots, m_n) = 4. $$
\end{theorem}

We remark that for the special cases of the principal and minimal strata, the large genus asymptotics of their Masur--Veech volumes were established earlier in \cite{CMZ18} and \cite{S18}, respectively. Furthermore, in \cite{S21}, Sauvaget proved the existence of a complete asymptotic expansion for these volumes, which depends only on the genus and the number of singularities. 

When a stratum is disconnected due to hyperelliptic or spin structures, the aforementioned ideas and methods can also be applied to study the Masur--Veech volume of each individual connected component; see, e.g.,~\cite{EOP08, CMSZ20}. It is worth noting that the spin components behave numerically similarly to connected strata in terms of large genus volume asymptotics, whereas the hyperelliptic components behave quite differently; see, e.g.,~\cite[p.~484]{EZ15}. Indeed, the Masur--Veech volumes and other geometric invariants of the hyperelliptic components can often be computed explicitly for all genera, due to the hyperelliptic double cover, which reduces the problem to studying the strata of quadratic differentials on the Riemann sphere; see the subsection below. 

\subsection{Quadratic differentials, $k$-differentials, cone metrics, and meromorphic differentials}
\label{subsec:k-diff}

In this section, we discuss several generalizations from holomorphic differentials to some other analogous structures. 

\subsubsection{$k$-differentials and $(1/k)$-translation surfaces} 

On a smooth complex curve $X$, a $k$-differential $\zeta$ is a section of the $k$-th power of the canonical bundle of $X$. Many properties of differential one-forms naturally generalize to $k$-differentials for higher $k$. For example, a $k$-differential $\zeta$ determines a $(1/k)$-translation structure, which allows translations and rotations by angles that are integer multiples of $2\pi / k$. Under the induced flat metric, a singularity of $\zeta$ of order $m$ has cone angle $2\pi (m+k)/k$. Note that as long as $m > -k$, this angle is well defined. Therefore, let $\nu$ be an integer partition of $k(2g-2)$ whose entries are all bigger than $-k$. Denote by $\mathcal H^k(\nu)$ the stratum of $k$-differentials parameterizing pairs $(X, \zeta)$ of type $\nu$, whose associated $(1/k)$-translation surfaces have finite area.

To relate back to differential one-forms, there exists a canonical cyclic cover $\pi\colon \widehat{X} \to X$ of degree $k$ such that $\pi^{*}\zeta = \omega^k$, where $\omega$ is a holomorphic one-form of type $\mu$ on $\widehat{X}$ determined by $\nu$, and is an eigenvector for a primitive $k$-th root of unity under the deck transformation of $\pi$. Therefore, up to the choice of a $k$-th root of unity, one can lift $\mathcal H^k(\nu)$ into the corresponding stratum of holomorphic one-forms $\mathcal H(\mu)$ as a subvariety, where the corresponding eigenspace of period coordinates in $\mathcal H(\mu)$ provides local coordinates for $\mathcal H^k(\nu)$. For instance, when $k=2$, these are the anti-invariant periods under the canonical double cover. We refer to \cite{BCGGMk} for a detailed introduction to the strata of $k$-differentials and their compactifications. 

We remark that the classification of connected components of the strata of $k$-differentials remains largely open for $k \geq 3$, due to the absence of the ${\rm GL}_2^{+}(\mathbb R)$-action and the associated vertical flow decomposition for $(1/k)$-translation surfaces, in contrast to the cases $k=1,2$; see~\cite{CG22} for partial results and related discussions. 

\subsubsection{Masur--Veech volumes of strata of quadratic differentials} 

As mentioned previously, the ${\rm GL}_2^{+}(\mathbb R)$-action is well defined on the strata of $k$-differentials only for $k=1,2$. Therefore, aside from the case of one-forms, the only other instance is that of quadratic differentials, whose strata we denote by $\mathcal Q(\nu)$. As in the case of one-forms, the strata of quadratic differentials may also be disconnected; see \cite{L08, CM14, CG22} for the classifications of their connected components.  

As explained above, we allow quadratic differentials to have simple poles, which correspond to cone angles equal to $\pi$, and the resulting half-translation surfaces still have finite area. Since the involution induced by the canonical double cover preserves the lattice $\mathbb Z \oplus \mathbb Z{\rm i} \subset \mathbb C$, the anti-invariant part of the relative cohomology $H^1_{-}(\widehat{X}, \{ \widehat{z}_i \}; \mathbb Z \oplus \mathbb Z{\rm i})$ spans a lattice in $H^1_{-}(\widehat{X}, \{ \widehat{z}_i \}; \mathbb C)$. We can normalize this lattice to have covolume one, thereby defining a Masur--Veech volume form on $\mathcal Q(\nu)$. When restricting to the hyperboloid $\mathcal Q(\nu)^{\le \frac{1}{2}}$ of half-translation surfaces with area bounded by $1/2$, the finiteness theorem of Masur and Veech, as reviewed in Theorem~\ref{thm:finite}, continues to hold, thus defining a finite Masur--Veech volume for the stratum $\mathcal Q(\nu)$. 

Similar to square-tiled surfaces, integral points in the strata of quadratic differentials $\mathcal Q(\nu)$ correspond to {\em pillowcase covers}, which are covers of $\mathbb C\mathbb P^1$ branched over four points. For instance, a quadratic differential with four simple poles on $\mathbb C\mathbb P^1$ can be obtained by folding a $2\times 1$ rectangle along its midline so that the left and right squares are identified, forming a pillowcase. In parallel with their result for Abelian differentials, Eskin and Okounkov~\cite{EO06} showed that the generating function of pillowcase covers is also quasimodular, providing a theoretical algorithm for computing the volume of $\mathcal Q(\nu)$. Following this line of work, Goujard~\cite{Gou16} computed explicit values for the Masur--Veech volumes of strata of quadratic differentials up to dimension~$10$. An explicit formula for the volumes of strata in genus~$0$, conjectured earlier by Kontsevich, was later proved by Athreya, Eskin, and Zorich~\cite[Theorem 1.1]{AEZ16} using the values of area Siegel--Veech constants in genus~$0$ from the work of Eskin, Kontsevich, and Zorich~\cite[Theorem 3]{EKZ14}, which relates these invariants to sums of Lyapunov exponents with respect to the Teichm\"uller geodesic flow. 

Despite these advances, the asymptotic analysis of pillowcase covers is considerably more involved than that of square-tiled surfaces. Consequently, although a similar intersection-theoretic volume formula is expected to hold for all strata of quadratic differentials (see~\cite[Conjecture 1.1]{CMS23}), the indirect proof of Theorem~\ref{thm:abelian-volume-int} has not yet been fully developed in this setting. Only in the special case where all periods are absolute—that is, when all singularity orders of the quadratic differential are odd—the area Hermitian metric approach remains applicable, as shown by M\"oller, Sauvaget, and the first-named author  in~\cite[Theorem~1.2]{CMS23}. 

Let $\mathcal Q(1^{4g-4+n}, -1^n)$ be the principal stratum consisting of quadratic differentials with simple zeros and poles, occupies a distinguished position among all strata of quadratic differentials, as it can be identified with an open subset of the cotangent bundle of $\mathcal M_{g,n}$. For this reason, several different approaches have been developed to compute the Masur–Veech volume of $\mathcal Q(1^{4g-4+n}, -1^n)$ and to study its asymptotic behavior for large $g$ and $n$. From the viewpoint of topological recursion, see \cite{the7} and \cite[Appendix]{CMS23}; for an expression in terms of $\psi$-class intersection numbers, which also appear in the computation of Weil–Petersson volumes, see \cite{DGZZ21}; and for the volume asymptotics, see \cite{CMS23, YZZ20, K22, A21}. In particular, Aggarwal proved in \cite[Theorem~1.7]{A21} the large genus limits for the volumes of $\mathcal Q(1^{4g-4+n}, -1^n)$ as originally conjectured in \cite{ADGZZ20}. 

It is worth noting that \cite{ADGZZ20} provides conjectural descriptions for the large genus limits of the Masur--Veech volumes of other strata of quadratic differentials, which remain largely open problems. Even for the strata of quadratic differentials with only odd singularities, the intersection-theoretic formula for their volumes in \cite[Theorem~1.2]{CMS23} cannot be evaluated recursively, since along the boundary some odd singularities may merge to form even ones. Recently, Duryev, Goujard, and Yakovlev \cite{DGY25} expressed the volumes of certain completed strata of quadratic differentials with only odd singularities as a sum over stable graphs, by counting ribbon graphs appearing in Kontsevich’s proof of Witten’s conjecture. These completed strata include some adjacent strata beyond the purely odd ones, which still need to be removed. 

We summarize some of the open problems related to the Masur--Veech volumes of strata of quadratic differentials as follows. 

\begin{problem}
 Prove the conjectural formula proposed in~\cite[Conjecture~1.1]{CMS23} for computing the Masur--Veech volumes of the strata of quadratic differentials via intersection theory. 
 
 Furthermore, derive recursion relations for the asymptotics of pillowcase covers and the associated intersection numbers, and compare their structures. 
 
 In addition, establish the large genus asymptotics of the Masur--Veech volumes of strata of quadratic differentials as predicted in~\cite[Conjecture~1]{ADGZZ20}. 
\end{problem}

\subsubsection{Volumes of strata of $k$-differentials} 

Now, consider the strata of $k$-differentials for $k \geq 3$. In this case, the ${\rm GL}_2^{+}(\mathbb{R})$-action does not commute with rotations by an angle $2\pi/k$. Moreover, the $k$-th root Abelian differential on the canonical $k$-cover, as well as the corresponding period coordinates of $\mathcal{H}^k(\mu)$, depend on the choice of a primitive $k$-th root of unity. Hence, even defining the Masur--Veech volume form for $\mathcal{H}^k(\mu)$ requires additional work. Nguyen proved that $\mathcal{H}^k(\mu)$ carries a volume form inducing a finite total volume on $\mathbb{P}\mathcal{H}^k(\mu)$, defined with respect to a choice of a primitive $k$-th root of unity $\zeta_k$; see~\cite[Theorem~1.1]{N22}. He further showed how this volume form changes when the choice of $\zeta_k$ varies, and in particular, that it is independent of the choice for $k = 3, 4, 6$; see~\cite[Remark 5.5]{N22}. 

Alternatively, for the cases $k = 3, 4, 6$, Engel~\cite{E21} studied elliptic orbifold covers corresponding to certain hexagon, square, and triangle tilings. Similar to the cases $k = 1, 2$, the large degree  asymptotics of the associated Hurwitz numbers provide an algorithm to compute the Masur--Veech volumes of the strata of cubic, quartic, and sextic differentials. We refer to~\cite[Section~5.2]{N22} for a comparison of these volume forms and to~\cite{KN21} for a related work. 

\subsubsection{Volumes of moduli spaces of flat surfaces with cone metrics} 

This story is particularly rich for flat cone metrics on the sphere of genus zero, where a cone metric is locally isometric to an open subset of a flat cone, e.g., triangulations of the sphere by Euclidean triangles. Note that, in general, a surface equipped with a cone metric does not necessarily correspond to a differential form, since the transition functions need not be translations and rotations by angles that are rational multiples of $2\pi$. Nevertheless, computing the volumes of the associated moduli spaces remains an active topic of research, whose study also connects to the hyperbolic geometric aspect of Weil--Petersson volumes; see, e.g.,~\cite{T98, M17, ES18, KN18, N24}; see also~\cite{KT99} for a related discussion from the perspective of the moduli space of spatial polygons. 

For general genus $g$, let $\mathcal M(a)$ denote the moduli space of cone metrics whose cone angles (modulo $2\pi$) are specified by
$a = (a_1,\ldots, a_n)$ with $a_i \in \mathbb R^{+}$ and $\sum_{i=1}^n a_i = 2g-2+n$. Each $\mathcal M(a)$ is diffeomorphic to $\mathcal M_{g,n}$, as shown by Troyanov \cite{T86}. Veech \cite{V93} introduced systems of local coordinates and constructed a measure on $\mathcal M(a)$. As mentioned in \S\ref{subsec:higherangles}, in a recent work, Sauvaget \cite{S24} discovered a rational piecewise polynomial which, up to an explicit normalization factor, determines the volume of $\mathcal M(a)$ for all rational vectors $a$ without integral coordinates. Moreover, the normalized volume formula extends continuously and is conjectured to compute the volume of $\mathcal M(a)$ for all $a$. 

A key insight in Sauvaget’s approach lies in explicit recursions for the Masur--Veech volumes of the strata of $k$-differentials with no integral singularities, together with the study of their volume asymptotics as $k \to \infty$. In this regime, the metrics induced by $k$-differentials behave like rational points in the moduli space of all cone metrics, and the large-$k$ asymptotics are analogous to approximating real numbers by rational numbers with large denominators. As remarked by Sauvaget, to fully resolve his conjecture, a compactification of the moduli space $\mathcal M(a)$---analogous to the multi-scale compactification of the strata of differentials in \cite{BCGGM1, BCGGMk, BCGGM2, CMZ24}---would play a crucial role; see \cite{GP17, GP20} for related discussions in genus one; however, such a construction is still missing in the literature for general $g$. 

\subsubsection{A quest for volumes of strata of meromorphic differentials} 

Finally, consider the strata of meromorphic differentials, which correspond to translation surfaces of {\em infinite} area. In this case, there is no hyperboloid of meromorphic differentials with bounded area inside the strata, and hence it remains unclear how to define the Masur--Veech volume form. In contrast, the multi-scale compactification can still be constructed for the strata of meromorphic differentials, and all tautological divisor classes, together with their intersection numbers, are well defined. 
In particular, Sauvaget \cite{S25} defined virtual volumes for strata of meromorphic differentials with only simple poles using the intersection numbers in \eqref{eq:MV-abelian}, and showed that these virtual volumes satisfy properties analogous to those in the holomorphic case. In general, it would be meaningful to explore flat-geometric interpretations of these intersection numbers for strata of meromorphic differentials with arbitrary pole orders---for instance, to determine which of them might correspond to volume-like invariants.  

\medskip

We summarize the open problems mentioned above as follows. 

\begin{problem}
For strata of $k$-differentials and cone metrics in general, prove the conjectural intersection-theoretic formula proposed in~\cite[Conjecture~1.6]{S24} for computing their Masur--Veech volumes, and analyze their large genus asymptotics. 

Furthermore, for strata of meromorphic differentials, define a meaningful volume form and identify flat-geometric interpretations of the intersection numbers arising from tautological divisor classes, analogous to those appearing in the volume–intersection formulae in the case of holomorphic differentials.
\end{problem}

\subsection{Linear subvarieties and Siegel--Veech constants}
\label{subsec:linear-SV}

A breakthrough result of Eskin, Mirzakhani, and Mohammadi establishes that every ${\rm GL}_2^{+}(\mathbb R)$-orbit closure in the strata of holomorphic differentials is locally defined by linear equations in the period coordinates, whose coefficients were later shown by Filip to be algebraic numbers; see \cite{EMM15, EM18, F16}. Hence, these orbit closures are also called {\em linear subvarieties}. 

\subsubsection{Masur--Veech volumes of linear subvarieties} 

Whenever a “unit cube” element is well defined in the subspace of period coordinates of a linear subvariety—for instance, when the defining linear equations are defined over $\mathbb Q$—we can define its Masur--Veech volume as before. Such examples include arithmetic Teichm\"uller curves generated by square-tiled surfaces, the strata of quadratic differentials lifted to the corresponding strata of Abelian differentials via the canonical double cover, and several recently discovered totally geodesic subvarieties; see~\cite{MMW17, EMMW20}. 

For an arithmetic Teichm\"uller curve, its volume can be determined by the number of square-tiled surfaces lying in the same ${\rm GL}_2^{+}(\mathbb R)$-orbit, which corresponds to the Hurwitz number of branched covers of the square torus with a single branch point and a prescribed ramification profile. A subtlety here is that the classification of connected components of the relevant Hurwitz space remains largely open, even in genus two. In other words, a given Hurwitz number may correspond to the total volume of several distinct arithmetic Teichm\"uller curves. We refer to~\cite{D24} for a comprehensive discussion of the connected components of arithmetic Teichm\"uller curves in genus two. As for the gothic locus, a totally geodesic subvariety discovered in~\cite{MMW17}, its Masur--Veech volume was computed in~\cite{TT20}, building on the formulae in~\cite{MTT20} for the Euler characteristics of Teichm\"uller curves contained in the gothic locus. 

If a linear subvariety is locally generated by absolute periods only, it is called {\em REL zero}. Such examples include the minimal strata $\mathcal H(2g-2)$, the strata of quadratic differentials with odd singularities, and the gothic locus. In this case, as explained in Section~\ref{subsec:intersection},
the Hermitian metric on the tautological bundle $\mathcal O(-1)$ induced by the area form defines a volume form, so that the top self-intersection number of $\xi = c_1(\mathcal O(1))$ on the compactified linear subvariety computes the corresponding volume. This approach for arbitrary REL-zero linear subvarieties was justified by Nguyen~\cite{N25}. For a Teichm\"uller curve, the degree of $\xi$ corresponds to its Euler characteristic; see~\cite[Formula~(15)]{CM12}. For the gothic locus, the top self-intersection number of $\xi$ was computed in~\cite{C22}, which also determines its Masur--Veech volume up to a normalization factor. 

\begin{problem}
Formulate a notion of Masur--Veech volumes for REL-nonzero linear subvarieties in the strata of holomorphic differentials, and relate it to the type of intersection numbers appearing in \eqref{eq:MV-abelian}.   
\end{problem}

\subsubsection{Counting geodesics on flat surfaces} 

Another family of invariants that coexist with the Masur--Veech volumes are the {\em Siegel--Veech constants}. Roughly speaking, they describe the quadratic growth rate, as the length tends to infinity, of the number of certain geodesic-related structures on a flat surface. For example, the number $N(L)$ of saddle connections joining two specified zeros and of length less than $L$ grows quadratically in $L$: 
\begin{align*}
N(L) \sim c L^2 
\end{align*}
with the leading coefficient---up to a normalization factor---defining the Siegel--Veech constant for such saddle connections. Besides saddle connections, one can also count closed geodesics of length less than $L$, or cylinders filled with closed geodesics, weighted by the ratio of their height to width. The phenomenon of quadratic growth can be easily illustrated in the simple case of counting line segments joining an integer lattice point to the origin and of length less than $L$, which corresponds to counting saddle connections of bounded length joining a marked point to itself on the unit-square torus. In general, this quadratic asymptotic behavior was systematically studied by Eskin and Masur~\cite{EM01}.

The computation of Siegel--Veech constants, much like that of the Masur--Veech volumes, is equally profound and intricate. To avoid overloading the present exposition, we briefly summarize the relationship between the two as follows. By shrinking the structures counted in the definition of the Siegel--Veech constants---for instance, by letting the zeros joined by a saddle connection coalesce, or by pinching the core curve of a cylinder---one approaches certain boundary strata of the moduli space of differentials. The volumes of these boundary neighborhoods, divided by the total volume of the moduli space and multiplied by combinatorial factors arising from the counting problem, yield recursive formulas for the Siegel--Veech constants in terms of the Masur--Veech volumes. This approach was explained in detail in the work of Eskin, Masur, and Zorich~\cite{EMM03}, where the corresponding boundary strata are referred to as {\em principal boundaries}. It has since found wide applications in the computation of Siegel--Veech constants for Abelian and quadratic differentials, as well as in the study of the geometry of compactifications of the corresponding moduli spaces; see~\cite{BG15, Gou15, AEZ16, CC19P, CGM25}. 

\subsubsection{Area Siegel--Veech constants as intersection numbers} 

Among all Siegel--Veech constants, the {\em area} Siegel--Veech constant $c_{\rm area}$---which counts cylinders of bounded length weighted by their modulus (i.e., height over width)---is particularly significant from the perspective of intersection theory. Intuitively, when one pinches the core curve of a cylinder to form a node, the local type of the resulting node is encoded by the cylinder modulus, in a manner analogous to counting intersection numbers with the appropriate multiplicities. 

A distinguished aspect of the area Siegel--Veech constant is the remarkable formula of Eskin, Kontsevich, and Zorich \cite[Theorem 1]{EKZ14} that relates the area Siegel--Veech constant $c_{\rm area}$ and the sum of Lyapunov exponents $L$ for any linear subvariety $\mathcal M$ in the stratum $\mathcal H(\mu)$ of Abelian differentials of type $\mu$:  
\begin{align*}
L (\mathcal M) = \frac{\pi^2}{3}\cdot c_{\rm area} (\mathcal M) + 
\kappa_{\mu}, 
\end{align*}
where $\kappa_\mu = \frac{1}{12} \sum_{i=1}^n \frac{m_i (m_i+2)}{m_i+1}$ is a constant determined by the zero type $\mu = (m_1,\ldots, m_n)$. The incarnation of this formula in algebraic geometry is Mumford's relation (also called the Noether formula): 
\begin{align*}
\lambda = \frac{\delta + \kappa}{12}, 
\end{align*}
where $\lambda$ is the first Chern class of the Hodge bundle over $\overline{\mathcal M}_g$, $\delta$ is the boundary divisor class, and $\kappa$ is the first Miller--Morita--Mumford class; see~\cite[p. 102]{M77}. 

The correspondence between these two formulae can be understood intuitively as follows. In special cases where the Hodge bundle splits into invariant subbundles under the Teichm\"uller geodesic flow, each invariant subbundle contributes to the corresponding Lyapunov exponents, and hence their sum corresponds to the first Chern class of the entire Hodge bundle. Moreover, pinching the core curve of each cylinder produces a nodal surface parameterized by the boundary of $\overline{\mathcal M}_g$. Finally, the constant $\kappa_\mu$ naturally arises from tautological divisor class relations on the moduli space of differentials of type~$\mu$. This correspondence can be justified rigorously in the case of Teichm\"uller curves, which also connects to the study of their slopes; see~\cite{CM12, CM14}. Therefore, it would be meaningful to establish an intersection-theoretic formula capable of computing the area Siegel--Veech constant for all linear subvarieties.  

Indeed, such a formula was speculated in \cite[Main Theorem]{K97} for the strata of holomorphic differentials and conjectured explicitly for all linear subvarieties in~\cite[Conjecture~4.3]{CMS23} as follows: 
\begin{align}
\label{eq:c_area-conjecture}
c_{\rm area}(\mathcal M) = -\frac{1}{4\pi^2} \frac{\int_{\mathbb P\overline{\mathcal M}}\xi^{a-2}\psi_1\cdots \psi_r \delta}{\int_{\mathbb P\overline{\mathcal M}}\xi^{a-1}\psi_1\cdots \psi_r}, 
\end{align}
where $a$ is the rank of the absolute period space of $\mathcal M$, $r$ is the rank of the (strictly) relative period space, $\delta$ denotes the boundary divisor of the multi-scale compactification $\mathbb P\overline{\mathcal M}$, $\xi$ is the universal line bundle class, and $\psi_1, \ldots, \psi_r$ are the $\psi$-classes associated with the marked zeros whose variations generate the relative period space of~$\mathcal M$. Note that the denominator in the conjectural formula has the same form as the intersection-theoretic expression computing the Masur--Veech volume in~\eqref{eq:MV-abelian}. This is natural, since the numerator measures the total contribution obtained by pinching the core curves of cylinders to the boundary intersection, and the ratio with the total volume of the linear subvariety yields the averaged quantity. Moreover, the area Siegel--Veech constant is intrinsically defined for~$\mathcal M$, independent of the choice of its volume normalization, and so is the expression in~\eqref{eq:c_area-conjecture}, where any such normalization cancels out between the numerator and the denominator. 

Besides Teichm\"uller curves, the conjectural formula~\eqref{eq:c_area-conjecture} is also known to hold for the strata of Abelian differentials~\cite{S18,CMSZ20, IS25}, the principal strata of quadratic differentials~\cite{CMS23}, the gothic locus~\cite{C22}, and all linear subvarieties of REL zero~\cite{CGM25}, which include in particular the strata of quadratic differentials with odd singularities. However, when a linear subvariety admits a free relative period, the conjectural formula remains open in general. 

It is worth noting that the curve class $\xi^{a-2}\psi_1\cdots \psi_r$ in \eqref{eq:c_area-conjecture} is expected to represent the differential form $\beta$ in Kontsevich's speculation, where it remains unknown in general whether $\beta$ is a rational cohomology class. On the other hand, $\beta$ appears to resemble the ``limit'' of the numerical classes of arithmetic Teichm\"uller curves of large degree, normalized in a suitable way. For instance, if the rational Picard group of $\mathbb P\overline{\mathcal M}$ were known (which, however, remains largely mysterious even for the strata of differentials), then comparing the intersection numbers of Teichm\"uller curves and the $\beta$ class with interior and boundary divisor classes could yield a proportionality relation between them. 

\medskip

We summarize these questions as follows. 

\begin{problem}
Study the principal boundary of a REL-nonzero linear subvariety and develop a recursion for its area Siegel--Veech constant in terms of volumes and intersection numbers. 

Alternatively, provide a justification of Kontsevich’s $\beta$ class via the intersection of divisor classes and the large degree  limit of Teichm\"uller curves. 

In either approach, establish the conjectural formula~\eqref{eq:c_area-conjecture} in full generality. 
\end{problem}

\subsubsection{Variants of Siegel--Veech constants} 

Finally, one can consider Siegel--Veech constants weighted by the area of cylinders to a different power as well as how these constants behave under covering constructions; see~\cite{ACMSS}. One can also inquire about Siegel--Veech-type invariants for flat surfaces other than translation surfaces. For $(1/k)$-translation surfaces (i.e., $k$-differentials) with $k \geq 3$, due to the absence of a ${\rm GL}_2^+(\mathbb R)$-action, only certain weak quadratic asymptotics in the sense of Ces\`aro means were established in~\cite{A25} for the cylinder and saddle-connection Siegel--Veech constants. For cone spheres whose cone angles are irrational multiples of~$\pi$, a generalization of the classical Siegel--Veech formula was developed in~\cite{F25} for saddle connections and closed geodesics. In the case of meromorphic differentials with poles (i.e., translation surfaces of infinite area), although the ${\rm GL}_2^+(\mathbb R)$-action is well defined on the strata of meromorphic differentials, it is generally not ergodic; see~\cite[Appendix~A]{B15}. In contrast, the strata of meromorphic differentials admit wall–chamber decompositions, where the counting of cylinders and saddle connections is expected to satisfy the wall-crossing formula of Kontsevich and Soibelman  \cite{KS08}, whose study is closely related to Bridgeland stability conditions, BPS invariants, and Donaldson--Thomas theory. As an incomplete overview of this rapidly developing topic, we refer the reader to~\cite{GMN13, BS15, HKK17, A23, KW24, HM25} for several representative results among many others. In higher dimensions, K3 surfaces generalize flat tori and also admit interesting geometric invariants that behave analogously to Siegel--Veech constants and Lyapunov exponents for flat surfaces; see \cite{F18, F20, AFL23}. 

We summarize the phenomena discussed above into the following relatively vague but long-term meaningful question.

\begin{problem}
    Compare the structures of Siegel--Veech-type invariant counting for holomorphic differentials, meromorphic differentials, $k$-differentials, and surfaces equipped with cone metrics, as well as their higher-dimensional analogs. 
    
    Investigate their intrinsic connections, and explore whether they exhibit structural analogies with other classical enumerative theories, such as Gromov--Witten theory, Donaldson--Thomas theory, and mirror symmetry. 
\end{problem}

\section{Common themes of both volumes}
\label{sec:common}

 We conclude this survey by summarizing several features that appear in the study of both Weil--Petersson and Masur--Veech volumes. 

First, the inputs of metrics and cone angles as well as the volume expressions and counting geodesics as intersection numbers clearly present strong similarities between the two volumes. In particular, the Weil--Petersson metric represented by the class $\kappa_1$
and the hermitian metric associated to the area of flat surfaces represented by the class $\xi$ reveal an intrinsic geometric idea commonly utilized in the computation of volumes through intersection numbers on both sides.  

Additionally, $\psi$-classes also appear in the computation of both volumes. On hyperbolic surfaces, as an $S^1$-bundle, each $\psi$-class can be interpreted as the hyperbolic geodesic boundary around the corresponding cusp. On flat surfaces, heuristically speaking (see Problem~\ref{prob:psi}), each $\psi$-class keeps track of the variation of the corresponding cone point. 

A closely related context for both volumes is Kontsevich's proof of Witten's conjecture (see \cite{Kont92}), where $\psi$-classes appear as polygonal boundaries surrounding the double poles of Strebel differentials (a special kind of quadratic differentials) under the associated flat metric; see also \cite{Mo11} for a description of the interpolation of the hyperbolic and flat perspectives that produce ribbon graphs as a key structure in Kontsevich's proof of Witten's conjecture. 

Furthermore, another proof of Witten's conjecture given by Pandharipande and Okounkov (see \cite{OP09}) crucially used Hurwitz numbers of branched covers of $\mathbb P^1$, while Hurwitz numbers of branch covers of tori count lattice points in the strata of holomorphic differentials whose large degree asymptotics determine the Masur--Veech volume.

Finally, as already mentioned, Sauvaget~\cite{Sauv} has recently provided an innovative way forward on the open question of the computation of Weil--Petersson volumes when a cone angle is greater than $2\pi$, proposing a definition of the volume of these moduli spaces as a limit of the Masur--Veech volumes of moduli spaces of $k$-differentials as $k\to \infty$. 
 
Recall that a $k$-differential on a Riemann surface with zeros or poles of order $>-k$ gives a flat metric of finite area on the surface with conical singularities of angle $2\pi(\frac{m}{k}+1)$ at a zero or pole of order $m>-k$.  Heuristically, by fixing the cone angles at $n$ points and considering $k$-differentials with simple zeros outside these points, the simple zeros form a small punctual negative curvature, which as $k\to \infty$  equidistribute to approximate a smooth metric of constant negative curvature. This heuristic convergence has already been used in theoretical physics where the Weil--Petersson measure is used to compute the partition functions in Jackiw--Teitelboim (JT) theory, a gravity theory with dilaton.

Together with the conjectured Masur--Veech volumes of the moduli spaces of $k$-differentials given recently in \cite[Conjecture 1.6]{S24}, 
this provides a conjectural description of the volume function more generally which Sauvaget shows corresponds with the Weil--Petersson volume for $\mathbf{L}\in i[0,2\pi)^n$ in the case that $-iL_j=\theta_j<\pi$ for at least $n-1$ angles.

\bibliographystyle{alpha}
\bibliography{biblio}

\newcommand{\etalchar}[1]{$^{#1}$}
\begin{thebibliography}{EMMW20}

\bibitem[ABC{\etalchar{+}}23]{the7}
J\o rgen~Ellegaard Andersen, Ga\"etan Borot, S\'everin Charbonnier, Vincent
  Delecroix, Alessandro Giacchetto, Danilo Lewa\'nski, and Campbell Wheeler.
\newblock Topological recursion for {M}asur-{V}eech volumes.
\newblock {\em J. Lond. Math. Soc. (2)}, 107(1):254--332, 2023.

\bibitem[AC96]{ACoCom}
Enrico Arbarello and Maurizio Cornalba.
\newblock Combinatorial and algebro-geometric cohomology classes on the moduli
  spaces of curves.
\newblock {\em J. Algebraic Geom.}, 5(4):705--749, 1996.

\bibitem[ACM{\etalchar{+}}26]{ACMSS}
David Aulicino, Aaron Calderon, Carlos Matheus, Nick Salter, and Martin
  Schmoll.
\newblock Siegel--{V}eech constants for cyclic covers of generic translation
  surfaces.
\newblock {\em J. Lond. Math. Soc. (2)}, 113(1):Paper No. e70413, 2026.

\bibitem[ADG{\etalchar{+}}20]{ADGZZ20}
Amol Aggarwal, Vincent Delecroix, \'Elise Goujard, Peter Zograf, and Anton
  Zorich.
\newblock Conjectural large genus asymptotics of {M}asur-{V}eech volumes and of
  area {S}iegel-{V}eech constants of strata of quadratic differentials.
\newblock {\em Arnold Math. J.}, 6(2):149--161, 2020.

\bibitem[AEZ16]{AEZ16}
Jayadev~S. Athreya, Alex Eskin, and Anton Zorich.
\newblock Right-angled billiards and volumes of moduli spaces of quadratic
  differentials on {$\mathbb C\rm P^1$}.
\newblock {\em Ann. Sci. \'Ec. Norm. Sup\'er. (4)}, 49(6):1311--1386, 2016.
\newblock With an appendix by Jon Chaika.

\bibitem[AFL23]{AFL23}
Jayadev~S. Athreya, Yu-Wei Fan, and Heather Lee.
\newblock Counting special {L}agrangian classes and semistable {M}ukai vectors
  for {K}3 surfaces.
\newblock {\em Geom. Dedicata}, 217(5):Paper No. 89, 21, 2023.

\bibitem[Agg20]{A20}
Amol Aggarwal.
\newblock Large genus asymptotics for volumes of strata of {A}belian
  differentials.
\newblock {\em J. Amer. Math. Soc.}, 33(4):941--989, 2020.
\newblock With an appendix by Anton Zorich.

\bibitem[Agg21]{A21}
Amol Aggarwal.
\newblock Large genus asymptotics for intersection numbers and principal strata
  volumes of quadratic differentials.
\newblock {\em Invent. Math.}, 226(3):897--1010, 2021.

\bibitem[All23]{A23}
Dylan G.~L. Allegretti.
\newblock On the wall-crossing formula for quadratic differentials.
\newblock {\em Int. Math. Res. Not. IMRN}, (9):8033--8077, 2023.

\bibitem[AM22]{AnMonk}
Nalini Anantharaman and Laura Monk.
\newblock A high-genus asymptotic expansion of {W}eil-{P}etersson volume
  polynomials.
\newblock {\em J. Math. Phys.}, 63(4):Paper No. 043502, 26, 2022.

\bibitem[AMN23]{AMN}
Lukas Anagnostou, Scott Mullane, and Paul Norbury.
\newblock Weil-{P}etersson volumes, stability conditions and wall-crossing,
  2023.

\bibitem[AN25]{ANoVol}
Lukas Anagnostou and Paul Norbury.
\newblock Volumes of moduli spaces of hyperbolic surfaces with cone points,
  2025.

\bibitem[Ayg25]{A25}
Juliet Aygun.
\newblock Counting geodesics on prime-order $k$-differentials, 2025.

\bibitem[BB24]{BB24}
Luca Battistella and Sebastian Bozlee.
\newblock Hyperelliptic {G}orenstein curves and logarithmic differentials.
\newblock {\em Forum Math. Sigma}, 12:Paper No. e127, 28, 2024.

\bibitem[BCG{\etalchar{+}}18]{BCGGM1}
Matt Bainbridge, Dawei Chen, Quentin Gendron, Samuel Grushevsky, and Martin
  M\"oller.
\newblock Compactification of strata of {A}belian differentials.
\newblock {\em Duke Math. J.}, 167(12):2347--2416, 2018.

\bibitem[BCG{\etalchar{+}}19]{BCGGMk}
Matt Bainbridge, Dawei Chen, Quentin Gendron, Samuel Grushevsky, and Martin
  M\"oller.
\newblock Strata of {$k$}-differentials.
\newblock {\em Algebr. Geom.}, 6(2):196--233, 2019.

\bibitem[BCG{\etalchar{+}}24]{BCGGM2}
Matt Bainbridge, Dawei Chen, Quentin Gendron, Samuel Grushevsky, and Martin
  Möller.
\newblock The moduli space of multi-scale differentials, 2024.

\bibitem[BDG22]{BDG22}
Frederik Benirschke, Benjamin Dozier, and Samuel Grushevsky.
\newblock Equations of linear subvarieties of strata of differentials.
\newblock {\em Geom. Topol.}, 26(6):2773--2830, 2022.

\bibitem[Ben23]{B23}
Frederik Benirschke.
\newblock The boundary of linear subvarieties.
\newblock {\em J. Eur. Math. Soc. (JEMS)}, 25(11):4521--4582, 2023.

\bibitem[Ber74]{BerSpa}
Lipman Bers.
\newblock Spaces of degenerating {R}iemann surfaces.
\newblock In {\em Discontinuous groups and {R}iemann surfaces ({P}roc. {C}onf.,
  {U}niv. {M}aryland, {C}ollege {P}ark, {M}d., 1973)}, volume No. 79 of {\em
  Ann. of Math. Stud.}, pages 43--55. Princeton Univ. Press, Princeton, NJ,
  1974.

\bibitem[BG15]{BG15}
Max Bauer and Elise Goujard.
\newblock Geometry of periodic regions on flat surfaces and associated
  {S}iegel-{V}eech constants.
\newblock {\em Geom. Dedicata}, 174:203--233, 2015.

\bibitem[Boi15]{B15}
Corentin Boissy.
\newblock Connected components of the strata of the moduli space of meromorphic
  differentials.
\newblock {\em Comment. Math. Helv.}, 90(2):255--286, 2015.

\bibitem[BS15]{BS15}
Tom Bridgeland and Ivan Smith.
\newblock Quadratic differentials as stability conditions.
\newblock {\em Publ. Math. Inst. Hautes \'Etudes Sci.}, 121:155--278, 2015.

\bibitem[CC19a]{CC19P}
Dawei Chen and Qile Chen.
\newblock Principal boundary of moduli spaces of abelian and quadratic
  differentials.
\newblock {\em Ann. Inst. Fourier (Grenoble)}, 69(1):81--118, 2019.

\bibitem[CC19b]{CC19}
Dawei Chen and Qile Chen.
\newblock Spin and hyperelliptic structures of log twisted differentials.
\newblock {\em Selecta Math. (N.S.)}, 25(2):Paper No. 20, 42, 2019.

\bibitem[CC24]{CC24}
Dawei Chen and Qile Chen.
\newblock Gorenstein contractions of multiscale differentials, 2024.

\bibitem[CG22]{CG22}
Dawei Chen and Quentin Gendron.
\newblock Towards a classification of connected components of the strata of
  {$k$}-differentials.
\newblock {\em Doc. Math.}, 27:1031--1100, 2022.

\bibitem[CGH{\etalchar{+}}22]{CGHMS}
Dawei Chen, Samuel Grushevsky, David Holmes, Martin Möller, and Johannes
  Schmitt.
\newblock A tale of two moduli spaces: logarithmic and multi-scale
  differentials, 2022.

\bibitem[CGM25]{CGM25}
Dawei Chen, Elise Goujard, and Martin Möller.
\newblock Area siegel--veech constants for affine invariant submanifolds of rel
  zero, 2025.

\bibitem[Che17]{C17}
Dawei Chen.
\newblock Degenerations of {A}belian differentials.
\newblock {\em J. Differential Geom.}, 107(3):395--453, 2017.

\bibitem[Che19]{C19}
Dawei Chen.
\newblock Tautological ring of strata of differentials.
\newblock {\em Manuscripta Math.}, 158(3-4):345--351, 2019.

\bibitem[Che22]{C22}
Dawei Chen.
\newblock Dynamical invariants and intersection theory on the flex and gothic
  loci.
\newblock {\em Eur. J. Math.}, 8:S42--S52, 2022.

\bibitem[Che24]{chek}
Leonid~O. Chekhov.
\newblock Fool's crowns, trumpets, and schwarzian, 2024.

\bibitem[CM12]{CM12}
Dawei Chen and Martin M\"oller.
\newblock Nonvarying sums of {L}yapunov exponents of {A}belian differentials in
  low genus.
\newblock {\em Geom. Topol.}, 16(4):2427--2479, 2012.

\bibitem[CM14]{CM14}
Dawei Chen and Martin M\"oller.
\newblock Quadratic differentials in low genus: exceptional and non-varying
  strata.
\newblock {\em Ann. Sci. \'Ec. Norm. Sup\'er. (4)}, 47(2):309--369, 2014.

\bibitem[CMS23]{CMS23}
Dawei Chen, Martin M\"oller, and Adrien Sauvaget.
\newblock Masur-{V}eech volumes and intersection theory: the principal strata
  of quadratic differentials.
\newblock {\em Duke Math. J.}, 172(9):1735--1779, 2023.

\bibitem[CMSZ20]{CMSZ20}
Dawei Chen, Martin M\"oller, Adrien Sauvaget, and Don Zagier.
\newblock Masur-{V}eech volumes and intersection theory on moduli spaces of
  {A}belian differentials.
\newblock {\em Invent. Math.}, 222(1):283--373, 2020.

\bibitem[CMZ18]{CMZ18}
Dawei Chen, Martin M\"oller, and Don Zagier.
\newblock Quasimodularity and large genus limits of {S}iegel-{V}eech constants.
\newblock {\em J. Amer. Math. Soc.}, 31(4):1059--1163, 2018.

\bibitem[CMZ24]{CMZ24}
Matteo Costantini, Martin M\"oller, and Jonathan Zachhuber.
\newblock The area is a good enough metric.
\newblock {\em Ann. Inst. Fourier (Grenoble)}, 74(3):1017--1059, 2024.

\bibitem[CW21]{CW21}
Dawei Chen and Alex Wright.
\newblock The {WYSIWYG} compactification.
\newblock {\em J. Lond. Math. Soc. (2)}, 103(2):490--515, 2021.

\bibitem[DGY25]{DGY25}
Eduard Duryev, Elise Goujard, and Ivan Yakovlev.
\newblock Volumes of odd strata of quadratic differentials, 2025.

\bibitem[DGZZ21]{DGZZ21}
Vincent Delecroix, \'Elise Goujard, Peter Zograf, and Anton Zorich.
\newblock Masur-{V}eech volumes, frequencies of simple closed geodesics, and
  intersection numbers of moduli spaces of curves.
\newblock {\em Duke Math. J.}, 170(12):2633--2718, 2021.

\bibitem[DN09]{DNoWei}
Norman Do and Paul Norbury.
\newblock Weil-{P}etersson volumes and cone surfaces.
\newblock {\em Geom. Dedicata}, 141:93--107, 2009.

\bibitem[Doz25]{D25}
Benjamin Dozier.
\newblock Compactifications of strata of differentials, 2025.

\bibitem[Dur24]{D24}
Eduard Duryev.
\newblock Teichm\"uller curves in genus two: square-tiled surfaces and modular
  curves.
\newblock {\em Geom. Topol.}, 28(9):3973--4056, 2024.

\bibitem[EKZ14]{EKZ14}
Alex Eskin, Maxim Kontsevich, and Anton Zorich.
\newblock Sum of {L}yapunov exponents of the {H}odge bundle with respect to the
  {T}eichm\"uller geodesic flow.
\newblock {\em Publ. Math. Inst. Hautes \'Etudes Sci.}, 120:207--333, 2014.

\bibitem[EM01]{EM01}
Alex Eskin and Howard Masur.
\newblock Asymptotic formulas on flat surfaces.
\newblock {\em Ergodic Theory Dynam. Systems}, 21(2):443--478, 2001.

\bibitem[EM18]{EM18}
Alex Eskin and Maryam Mirzakhani.
\newblock Invariant and stationary measures for the {${\rm SL}(2,\mathbb R)$}
  action on moduli space.
\newblock {\em Publ. Math. Inst. Hautes \'Etudes Sci.}, 127:95--324, 2018.

\bibitem[EMM15]{EMM15}
Alex Eskin, Maryam Mirzakhani, and Amir Mohammadi.
\newblock Isolation, equidistribution, and orbit closures for the {${\rm
  SL}(2,\mathbb R)$} action on moduli space.
\newblock {\em Ann. of Math. (2)}, 182(2):673--721, 2015.

\bibitem[EMMW20]{EMMW20}
Alex Eskin, Curtis~T. McMullen, Ronen~E. Mukamel, and Alex Wright.
\newblock Billiards, quadrilaterals and moduli spaces.
\newblock {\em J. Amer. Math. Soc.}, 33(4):1039--1086, 2020.

\bibitem[EMZ03]{EMM03}
Alex Eskin, Howard Masur, and Anton Zorich.
\newblock Moduli spaces of abelian differentials: the principal boundary,
  counting problems, and the {S}iegel-{V}eech constants.
\newblock {\em Publ. Math. Inst. Hautes \'Etudes Sci.}, (97):61--179, 2003.

\bibitem[Eng21]{E21}
Philip Engel.
\newblock Hurwitz theory of elliptic orbifolds, {I}.
\newblock {\em Geom. Topol.}, 25(1):229--274, 2021.

\bibitem[EO01]{EO01}
Alex Eskin and Andrei Okounkov.
\newblock Asymptotics of numbers of branched coverings of a torus and volumes
  of moduli spaces of holomorphic differentials.
\newblock {\em Invent. Math.}, 145(1):59--103, 2001.

\bibitem[EO06]{EO06}
Alex Eskin and Andrei Okounkov.
\newblock Pillowcases and quasimodular forms.
\newblock In {\em Algebraic geometry and number theory}, volume 253 of {\em
  Progr. Math.}, pages 1--25. Birkh\"auser Boston, Boston, MA, 2006.

\bibitem[EOP08]{EOP08}
Alex Eskin, Andrei Okounkov, and Rahul Pandharipande.
\newblock The theta characteristic of a branched covering.
\newblock {\em Adv. Math.}, 217(3):873--888, 2008.

\bibitem[ES18]{ES18}
Philip Engel and Peter Smillie.
\newblock The number of convex tilings of the sphere by triangles, squares, or
  hexagons.
\newblock {\em Geom. Topol.}, 22(5):2839--2864, 2018.

\bibitem[ET23]{ETu2Dd}
Lorenz Eberhardt and Gustavo~J. Turiaci.
\newblock 2d dilaton gravity and the weil-petersson volumes with conical
  defects, 2023.

\bibitem[EZ15]{EZ15}
Alex Eskin and Anton Zorich.
\newblock Volumes of strata of {A}belian differentials and {S}iegel-{V}eech
  constants in large genera.
\newblock {\em Arnold Math. J.}, 1(4):481--488, 2015.

\bibitem[Fil16]{F16}
Simion Filip.
\newblock Splitting mixed {H}odge structures over affine invariant manifolds.
\newblock {\em Ann. of Math. (2)}, 183(2):681--713, 2016.

\bibitem[Fil18]{F18}
Simion Filip.
\newblock Families of {K}3 surfaces and {L}yapunov exponents.
\newblock {\em Israel J. Math.}, 226(1):29--69, 2018.

\bibitem[Fil20]{F20}
Simion Filip.
\newblock Counting special {L}agrangian fibrations in twistor families of {K}3
  surfaces.
\newblock {\em Ann. Sci. \'Ec. Norm. Sup\'er. (4)}, 53(3):713--750, 2020.
\newblock With an appendix by Nicolas Bergeron and Carlos Matheus.

\bibitem[FP18]{FP18}
Gavril Farkas and Rahul Pandharipande.
\newblock The moduli space of twisted canonical divisors.
\newblock {\em J. Inst. Math. Jussieu}, 17(3):615--672, 2018.

\bibitem[Fu25]{F25}
Kai Fu.
\newblock Siegel-veech measures of convex flat cone spheres, 2025.

\bibitem[Gen18]{G18}
Quentin Gendron.
\newblock The {D}eligne-{M}umford and the incidence variety compactifications
  of the strata of {$\Omega\mathcal M_g$}.
\newblock {\em Ann. Inst. Fourier (Grenoble)}, 68(3):1169--1240, 2018.

\bibitem[GMN13]{GMN13}
Davide Gaiotto, Gregory~W. Moore, and Andrew Neitzke.
\newblock Wall-crossing, {H}itchin systems, and the {WKB} approximation.
\newblock {\em Adv. Math.}, 234:239--403, 2013.

\bibitem[Gol86]{GolInv}
William~M. Goldman.
\newblock Invariant functions on {L}ie groups and {H}amiltonian flows of
  surface group representations.
\newblock {\em Invent. Math.}, 85(2):263--302, 1986.

\bibitem[Gou15]{Gou15}
Elise Goujard.
\newblock Siegel-{V}eech constants for strata of moduli spaces of quadratic
  differentials.
\newblock {\em Geom. Funct. Anal.}, 25(5):1440--1492, 2015.

\bibitem[Gou16]{Gou16}
Elise Goujard.
\newblock Volumes of strata of moduli spaces of quadratic differentials:
  getting explicit values.
\newblock {\em Ann. Inst. Fourier (Grenoble)}, 66(6):2203--2251, 2016.

\bibitem[GP17]{GP17}
Selim Ghazouani and Luc Pirio.
\newblock Moduli spaces of flat tori with prescribed holonomy.
\newblock {\em Geom. Funct. Anal.}, 27(6):1289--1366, 2017.

\bibitem[GP20]{GP20}
S\'elim Ghazouani and Luc Pirio.
\newblock Moduli spaces of flat tori and elliptic hypergeometric functions.
\newblock {\em M\'em. Soc. Math. Fr. (N.S.)}, (164):viii+183, 2020.

\bibitem[Gru01]{Grushevsky}
Samuel Grushevsky.
\newblock An explicit upper bound for {W}eil-{P}etersson volumes of the moduli
  spaces of punctured {R}iemann surfaces.
\newblock {\em Math. Ann.}, 321(1):1--13, 2001.

\bibitem[GS24]{gonchar}
Alexander~B. Goncharov and Zhe Sun.
\newblock Exponential volumes of moduli spaces of hyperbolic surfaces, 2024.

\bibitem[Gu{\'{e}}16]{G16}
J{\'{e}}r{\'{e}}my Gu{\'{e}}r{\'{e}}.
\newblock A generalization of the double ramification cycle via log-geometry,
  2016.

\bibitem[Har74]{HarCha}
William Harvey.
\newblock Chabauty spaces of discrete groups.
\newblock In {\em Discontinuous groups and {R}iemann surfaces ({P}roc. {C}onf.,
  {U}niv. {M}aryland, {C}ollege {P}ark, {M}d., 1973)}, volume No. 79 of {\em
  Ann. of Math. Stud.}, pages 239--246. Princeton Univ. Press, Princeton, NJ,
  1974.

\bibitem[Has03]{HasMod}
Brendan Hassett.
\newblock Moduli spaces of weighted pointed stable curves.
\newblock {\em Adv. Math.}, 173(2):316--352, 2003.

\bibitem[HK14]{HKoAna}
John~H. Hubbard and Sarah Koch.
\newblock An analytic construction of the {D}eligne-{M}umford compactification
  of the moduli space of curves.
\newblock {\em J. Differential Geom.}, 98(2):261--313, 2014.

\bibitem[HKK17]{HKK17}
Fabian Haiden, Ludmil Katzarkov, and Maxim Kontsevich.
\newblock Flat surfaces and stability structures.
\newblock {\em Publ. Math. Inst. Hautes \'Etudes Sci.}, 126:247--318, 2017.

\bibitem[HM25]{HM25}
Johannes Horn and Martin Möller.
\newblock Wall-crossing formulas via spectral networks, 2025.

\bibitem[HT25a]{HideThomas}
Will Hide and Joe Thomas.
\newblock Large-{$n$} asymptotics for {W}eil-{P}etersson volumes of moduli
  spaces of bordered hyperbolic surfaces.
\newblock {\em Comm. Math. Phys.}, 406(9):Paper No. 203, 41, 2025.

\bibitem[HT25b]{huangtel}
Yi~Huang and Ivan Telpukhovskiy.
\newblock Moduli spaces of open strings have polylogarithmic mirzakhani
  volumes, 2025.

\bibitem[Hua25]{Huang}
Xuanyu Huang.
\newblock Asymptotic coefficients of weil-petersson volumes in the large genus,
  2025.

\bibitem[Kaz22]{K22}
Maxim Kazarian.
\newblock Recursion for {M}asur-{V}eech volumes of moduli spaces of quadratic
  differentials.
\newblock {\em J. Inst. Math. Jussieu}, 21(4):1471--1476, 2022.

\bibitem[Kim20]{Kimura}
Yusuke Kimura.
\newblock J{T} gravity and the asymptotic {W}eil-{P}etersson volume.
\newblock {\em Phys. Lett. B}, 811:135989, 6, 2020.

\bibitem[KN18]{KN18}
Vincent Koziarz and Duc-Manh Nguyen.
\newblock Complex hyperbolic volume and intersection of boundary divisors in
  moduli spaces of pointed genus zero curves.
\newblock {\em Ann. Sci. \'Ec. Norm. Sup\'er. (4)}, 51(6):1549--1597, 2018.

\bibitem[KN21]{KN21}
Vincent Koziarz and Duc-Manh Nguyen.
\newblock Variation of {H}odge structure and enumerating tilings of surfaces by
  triangles and squares.
\newblock {\em J. \'Ec. polytech. Math.}, 8:831--857, 2021.

\bibitem[Kon92]{Kont92}
Maxim Kontsevich.
\newblock Intersection theory on the moduli space of curves and the matrix
  {A}iry function.
\newblock {\em Comm. Math. Phys.}, 147(1):1--23, 1992.

\bibitem[Kon97]{K97}
Maxim Kontsevich.
\newblock Lyapunov exponents and {H}odge theory.
\newblock In {\em The mathematical beauty of physics ({S}aclay, 1996)},
  volume~24 of {\em Adv. Ser. Math. Phys.}, pages 318--332. World Sci. Publ.,
  River Edge, NJ, 1997.

\bibitem[KS08]{KS08}
Maxim Kontsevich and Yan Soibelman.
\newblock Stability structures, motivic donaldson-thomas invariants and cluster
  transformations, 2008.

\bibitem[KT99]{KT99}
Yasuhiko Kamiyama and Michishige Tezuka.
\newblock Symplectic volume of the moduli space of spatial polygons.
\newblock {\em J. Math. Kyoto Univ.}, 39(3):557--575, 1999.

\bibitem[KW24]{KW24}
Omar Kidwai and Nicholas~J. Williams.
\newblock {D}onaldson-{T}homas invariants for the {B}ridgeland-{S}mith
  correspondence, 2024.

\bibitem[KZ03]{KZ03}
Maxim Kontsevich and Anton Zorich.
\newblock Connected components of the moduli spaces of {A}belian differentials
  with prescribed singularities.
\newblock {\em Invent. Math.}, 153(3):631--678, 2003.

\bibitem[Lan08]{L08}
Erwan Lanneau.
\newblock Connected components of the strata of the moduli spaces of quadratic
  differentials.
\newblock {\em Ann. Sci. \'Ec. Norm. Sup\'er. (4)}, 41(1):1--56, 2008.

\bibitem[LW24]{LipWright}
Michael Lipnowski and Alex Wright.
\newblock Towards optimal spectral gaps in large genus.
\newblock {\em Annals of Probability}, 52(2):545--575, 2024.
\newblock Published March 2024.

\bibitem[Mas76]{MasExt}
Howard Masur.
\newblock Extension of the {W}eil-{P}etersson metric to the boundary of
  {T}eichmuller space.
\newblock {\em Duke Math. J.}, 43(3):623--635, 1976.

\bibitem[Mas82]{M82}
Howard Masur.
\newblock Interval exchange transformations and measured foliations.
\newblock {\em Ann. of Math. (2)}, 115(1):169--200, 1982.

\bibitem[McM17]{M17}
Curtis~T. McMullen.
\newblock The {G}auss-{B}onnet theorem for cone manifolds and volumes of moduli
  spaces.
\newblock {\em Amer. J. Math.}, 139(1):261--291, 2017.

\bibitem[McO88]{McOPoi}
Robert~C. McOwen.
\newblock Point singularities and conformal metrics on {R}iemann surfaces.
\newblock {\em Proc. Amer. Math. Soc.}, 103(1):222--224, 1988.

\bibitem[Mir07a]{MirSim}
Maryam Mirzakhani.
\newblock Simple geodesics and {W}eil-{P}etersson volumes of moduli spaces of
  bordered {R}iemann surfaces.
\newblock {\em Invent. Math.}, 167(1):179--222, 2007.

\bibitem[Mir07b]{MirWei}
Maryam Mirzakhani.
\newblock Weil-{P}etersson volumes and intersection theory on the moduli space
  of curves.
\newblock {\em J. Amer. Math. Soc.}, 20(1):1--23, 2007.

\bibitem[Mir13]{MirGrowth}
Maryam Mirzakhani.
\newblock Growth of {W}eil-{P}etersson volumes and random hyperbolic surfaces
  of large genus.
\newblock {\em J. Differential Geom.}, 94(2):267--300, 2013.

\bibitem[MMW17]{MMW17}
Curtis~T. McMullen, Ronen~E. Mukamel, and Alex Wright.
\newblock Cubic curves and totally geodesic subvarieties of moduli space.
\newblock {\em Ann. of Math. (2)}, 185(3):957--990, 2017.

\bibitem[Mon11]{Mo11}
Gabriele Mondello.
\newblock Riemann surfaces with boundary and natural triangulations of the
  {T}eichm\"uller space.
\newblock {\em J. Eur. Math. Soc. (JEMS)}, 13(3):635--684, 2011.

\bibitem[MP19]{MirPetri}
Maryam Mirzakhani and Bram Petri.
\newblock Lengths of closed geodesics on random surfaces of large genus.
\newblock {\em Comment. Math. Helv.}, 94(4):869--889, 2019.

\bibitem[MTT20]{MTT20}
Martin M\"oller and David Torres-Teigell.
\newblock Euler characteristics of {G}othic {T}eichm\"uller curves.
\newblock {\em Geom. Topol.}, 24(3):1149--1210, 2020.

\bibitem[Mum77]{M77}
David Mumford.
\newblock Stability of projective varieties.
\newblock {\em Enseign. Math. (2)}, 23(1-2):39--110, 1977.

\bibitem[MUW21]{MMW21}
Martin M\"oller, Martin Ulirsch, and Annette Werner.
\newblock Realizability of tropical canonical divisors.
\newblock {\em J. Eur. Math. Soc. (JEMS)}, 23(1):185--217, 2021.

\bibitem[MW17]{MW17}
Maryam Mirzakhani and Alex Wright.
\newblock The boundary of an affine invariant submanifold.
\newblock {\em Invent. Math.}, 209(3):927--984, 2017.

\bibitem[MZ00]{ManinZograf}
Yuri~I. Manin and Peter Zograf.
\newblock Invertible cohomological field theories and {W}eil-{P}etersson
  volumes.
\newblock {\em Ann. Inst. Fourier (Grenoble)}, 50(2):519--535, 2000.

\bibitem[MZ15]{MirZog}
Maryam Mirzakhani and Peter Zograf.
\newblock Towards large genus asymptotics of intersection numbers on moduli
  spaces of curves.
\newblock {\em Geom. Funct. Anal.}, 25(4):1258--1289, 2015.

\bibitem[Ngu22]{N22}
Duc-Manh Nguyen.
\newblock Volume forms on moduli spaces of {$d$}-differentials.
\newblock {\em Geom. Topol.}, 26(7):3173--3220, 2022.

\bibitem[Ngu24]{N24}
Duc-Manh Nguyen.
\newblock Intersection theory and volumes of moduli spaces of flat metrics on
  the sphere.
\newblock {\em Geom. Dedicata}, 218(2):Paper No. 32, 42, 2024.
\newblock With an appendix by Nguyen and Vincent Koziarz.

\bibitem[Ngu25]{N25}
Duc-Manh Nguyen.
\newblock On the volumes of linear subvarieties in moduli spaces of
  projectivized {A}belian differentials.
\newblock {\em Math. Ann.}, 391(1):937--964, 2025.

\bibitem[NN98]{NNAWei}
Marjatta N\"a\"at\"anen and Toshihiro Nakanishi.
\newblock Weil-{P}etersson areas of the moduli spaces of tori.
\newblock {\em Results Math.}, 33(1-2):120--133, 1998.

\bibitem[OP09]{OP09}
A.~Okounkov and R.~Pandharipande.
\newblock Gromov-{W}itten theory, {H}urwitz numbers, and matrix models.
\newblock In {\em Algebraic geometry---{S}eattle 2005. {P}art 1}, volume 80,
  Part 1 of {\em Proc. Sympos. Pure Math.}, pages 325--414. Amer. Math. Soc.,
  Providence, RI, 2009.

\bibitem[Pen92]{PenDiff}
R.~C. Penner.
\newblock Weil-{P}etersson volumes.
\newblock {\em J. Differential Geom.}, 35(3):559--608, 1992.

\bibitem[Sau18]{S18}
Adrien Sauvaget.
\newblock Volumes and {S}iegel-{V}eech constants of {$\mathcal{H}(2g-2)$} and
  {H}odge integrals.
\newblock {\em Geom. Funct. Anal.}, 28(6):1756--1779, 2018.

\bibitem[Sau21]{S21}
Adrien Sauvaget.
\newblock The large genus asymptotic expansion of {M}asur-{V}eech volumes.
\newblock {\em Int. Math. Res. Not. IMRN}, (20):15894--15910, 2021.

\bibitem[Sau24a]{Sauv}
Adrien Sauvaget.
\newblock A flat perspective on moduli spaces of hyperbolic surfaces, 2024.

\bibitem[Sau24b]{S24}
Adrien Sauvaget.
\newblock Volumes of moduli spaces of flat surfaces, 2024.

\bibitem[Sau25]{S25}
Adrien Sauvaget.
\newblock Virtual volumes of strata of meromorphic differentials with simple
  poles, 2025.

\bibitem[ST01]{STrInt}
Georg Schumacher and Stefano Trapani.
\newblock Estimates of {W}eil-{P}etersson volumes via effective divisors.
\newblock {\em Comm. Math. Phys.}, 222(1):1--7, 2001.

\bibitem[ST05]{STrVar}
Georg Schumacher and Stefano Trapani.
\newblock Variation of cone metrics on {R}iemann surfaces.
\newblock {\em J. Math. Anal. Appl.}, 311(1):218--230, 2005.

\bibitem[ST11]{STrWei}
Georg Schumacher and Stefano Trapani.
\newblock Weil-{P}etersson geometry for families of hyperbolic conical
  {R}iemann surfaces.
\newblock {\em Michigan Math. J.}, 60(1):3--33, 2011.

\bibitem[Thu98]{T98}
William~P. Thurston.
\newblock Shapes of polyhedra and triangulations of the sphere.
\newblock In {\em The {E}pstein birthday schrift}, volume~1 of {\em Geom.
  Topol. Monogr.}, pages 511--549. Geom. Topol. Publ., Coventry, 1998.

\bibitem[Tro86]{T86}
Marc Troyanov.
\newblock Les surfaces euclidiennes \`a{} singularit\'es coniques.
\newblock {\em Enseign. Math. (2)}, 32(1-2):79--94, 1986.

\bibitem[TT20]{TT20}
David Torres-Teigell.
\newblock Masur-{V}eech volume of the gothic locus.
\newblock {\em J. Lond. Math. Soc. (2)}, 102(1):405--436, 2020.

\bibitem[TT22]{TT22}
Michael Temkin and Ilya Tyomkin.
\newblock Reduction and lifting problem for differential forms on {B}erkovich
  curves.
\newblock {\em Adv. Math.}, 397:Paper No. 108208, 23, 2022.

\bibitem[TWZ06]{TWZGen}
Ser~Peow Tan, Yan~Loi Wong, and Ying Zhang.
\newblock Generalizations of {M}c{S}hane's identity to hyperbolic
  cone-surfaces.
\newblock {\em J. Differential Geom.}, 72(1):73--112, 2006.

\bibitem[Vee82]{V82}
William~A. Veech.
\newblock Gauss measures for transformations on the space of interval exchange
  maps.
\newblock {\em Ann. of Math. (2)}, 115(1):201--242, 1982.

\bibitem[Vee93]{V93}
William~A. Veech.
\newblock Flat surfaces.
\newblock {\em Amer. J. Math.}, 115(3):589--689, 1993.

\bibitem[vIS25]{IS25}
Jan-Willem van Ittersum and Adrien Sauvaget.
\newblock Cylinder counts and spin refinement of area {S}iegel-{V}eech
  constants.
\newblock {\em Comment. Math. Helv.}, 100(3):559--617, 2025.

\bibitem[Wei58]{WeiMod}
Andr\'e Weil.
\newblock Modules des surfaces de {R}iemann.
\newblock In {\em S\'eminaire {B}ourbaki; 10e ann\'ee: 1957/1958}, page~7.
  Secr\'etariat math\'ematique, Paris, 1958.
\newblock Textes des conf\'erences; Expos\'es 152 \`a{} 168; 2e \'ed
  corrig\'ee, Expos\'e{} 168.

\bibitem[Wit91]{WitQua}
Edward Witten.
\newblock On quantum gauge theories in two dimensions.
\newblock {\em Comm. Math. Phys.}, 141(1):153--209, 1991.

\bibitem[Wol85]{WolWei}
Scott Wolpert.
\newblock On the {W}eil-{P}etersson geometry of the moduli space of curves.
\newblock {\em Amer. J. Math.}, 107(4):969--997, 1985.

\bibitem[WX22]{WuXue316}
Yunhui Wu and Yuhao Xue.
\newblock Random hyperbolic surfaces of large genus have first eigenvalues
  greater than $3/16 - \varepsilon$.
\newblock {\em Geometric and Functional Analysis}, 33(2), 2022.

\bibitem[YZZ20]{YZZ20}
Di~Yang, Don Zagier, and Youjin Zhang.
\newblock Masur-{V}eech volumes of quadratic differentials and their
  asymptotics.
\newblock {\em J. Geom. Phys.}, 158:103870, 12, 2020.

\bibitem[Zog20]{Zog}
Peter Zograf.
\newblock On the large genus asymptotics of {W}eil-{P}etersson volumes, 2020.

\end{thebibliography}

\end{document}